\documentclass[11pt]{amsart}

\usepackage{amsmath,amssymb,mathrsfs} 
\usepackage[
colorlinks,linkcolor=blue,citecolor=blue,urlcolor=blue]{hyperref}
\usepackage{times}
\usepackage{gensymb}
\usepackage[all]{xy}

\theoremstyle{plain}
    \newtheorem{thm}{Theorem}[section]
    \newtheorem{ppn}[thm]{Proposition}
    \newtheorem{lem}[thm]{Lemma}
    \newtheorem{cor}[thm]{Corollary}
    \newtheorem{conj}[thm]{Conjecture}
\theoremstyle{definition}
    \newtheorem{dfn}[thm]{Definition}
    
\theoremstyle{remark}
    \newtheorem{rmk}[thm]{Remark}    
    \newtheorem{epl}[thm]{Example}

\numberwithin{equation}{section}

\allowdisplaybreaks

\def\C{\mathbb{C}}\def\Q{\mathbb{Q}}\def\R{\mathbb{R}}\def\Z{\mathbb{Z}}\def\F{\mathbb{F}}
\def\ol#1{\overline{#1}}\def\wt#1{\widetilde{#1}}\def\wh#1{\widehat{#1}}
\def\os#1#2{\overset{#1}{#2}}
\def\ot{\otimes}\def\ra{\rightarrow}
\def\a{\alpha}\def\b{\beta}\def\g{\gamma}\def\d{\delta}\def\pd{\partial}\def\e{\varepsilon}
\def\io{\iota}\def\l{\lambda}\def\k{\kappa}\def\z{\zeta}\def\x{\xi}\def\y{\eta}
\def\o{\omega}\def\O{\Omega}

\def\D{\Delta}

\def\G{\Gamma}

\def\L{\Lambda}\def\s{\sigma}\def\t{\tau}

\def\Hom{\operatorname{Hom}}

\def\Ker{\operatorname{Ker}}
\def\Coker{\operatorname{Coker}}

\def\Re{\operatorname{Re}}
\def\Im{\operatorname{Im}}
\def\rank{\operatorname{rank}}
\def\Res{\operatorname{Res}}

\newcommand{\GL}{\operatorname{GL}}
\newcommand{\tr}{\operatorname{tr}}
\newcommand{\cyc}{\operatorname{cyc}}

\newcommand\sD{\mathscr{D}}
\newcommand\sL{\mathscr{L}}

\newcommand\sX{\mathscr{X}}

\newcommand\be{\mathbf{e}}

\newcommand{\A}{{\mathbb{A}}}

\newcommand{\lra}{\longrightarrow}

\renewcommand{\P}{\mathbb{P}}

\newcommand\n{\nu}\newcommand\m{\mu}
\newcommand\vp{\varphi}
\newcommand\inj{\hookrightarrow}

\newcommand\ck{\wh{\k^\times}}
\newcommand\0{\circ}

\newcommand\Fr{\mathrm{Fr}}

\newcommand\ab{\mathrm{ab}}
\newcommand\et{\mathrm{\acute{e} t}}
\newcommand\Gal{\mathrm{Gal}}

\def\sp{\mathrm{sp}}

\def\bmat#1{\begin{bmatrix}#1\end{bmatrix}}

\newcommand\FFF[5]{{}_{#1}F_{#2}\left(\genfrac{}{}{0pt}{}{#3}{#4};#5\right)}

\begin{document}

\title{On the adelic Gaussian hypergeometric function}
\author{Masanori Asakura \and Noriyuki Otsubo}

\address{Department of Mathematics, Hokkaido University, Sapporo, 060-0810 Japan}
\email{asakura@math.sci.hokudai.ac.jp}

\address{Department of Mathematics and Informatics, Chiba University, Chiba, 263-8522 Japan}
\email{otsubo@math.s.chiba-u.ac.jp}

\begin{abstract}
We define the adelic hypergeometric function of special Gaussian type by means of a tower of hypergeometric curves. This function takes values in an adelic completed group ring and interpolates all the hypergeometric functions of the same type over all finite fields. 
It specializes at the unit argument to the adelic beta function of Ihara and Anderson. We prove some transformation formulas and a summation formula for the adelic hypergeometric function, which are known classically for complex hypergeometric functions.   
\end{abstract}

\date{\today.}
\subjclass[2020]{{\bf 33C05}, 11G30, 11S80}
\keywords{Hypergeometric functions, Adelic beta function}

\thanks{Research of the first author supported by JSPS KAKENHI Grant JP23K03025. 
Research of the second author supported by JSPS KAKENHI Grant JP24K06682.}

\maketitle

\section{Introduction}

The Gaussian hypergeometric functions $F(a,b;c;\l)$ with complex parameters $a$, $b$ and $c$ and a variable $\l$ (see Section \ref{s2.4}) is a most well-studied family of special functions. 
In this paper, concentrating on the case where $c=1$, we define the adelic analogue of $F(a,b;1;\l)$ 
and establish its basic properties. 

Our construction is a generalization of the definition of the adelic beta function due to Ihara \cite{ihara} and Anderson \cite{anderson}. 
Recall that by the Euler-Gauss summation formula, the hypergeometric function specializes at $\l=1$ to the beta function as
\begin{equation}\label{e0.1}
F(a,b;1;1)= -\frac{1}{2\pi i} \frac{(1-\be(a))(1-\be(b))}{1-\be(a+b)}B(a,b),
\end{equation}
where $\be(t)=\exp(2\pi it)$ (see \eqref{euler-gauss}). 
Over a finite field $\k$, there are analogous functions 
$F(\a,\b;\g;\l)$ with parameters $\a$, $\b$ and $\g$ being multiplicative characters of $\k$ and $\l \in \k$, taking values in a cyclotomic field (see Section \ref{s4.1}). 
When $\l=\e$ (the unit character), their special values at $\l=1$ reduce to the Jacobi sums: 
\begin{equation}\label{e0.2}
F(\a,\b;\e;1)= j(\a,\b)
\end{equation}
if $\a, \b \ne \e$ (see \eqref{euler-gauss-fin}). 
The similarity between the beta function and the Jacobi sums can be explained by that they are different realizations of a same family of motives associated with the Fermat curves: the former interpolates the complex periods and the latter are the Frobenius traces on the \'etale cohomology. 

In this paper, we consider the hypergeometric curves with the affine equation
\[X_{N,\l} : (1-x^N)(1-y^N)=\l x^Ny^N\]
(see Section \ref{s2.1}). 
Note that it specializes at $\l=1$ to the Fermat curve of degree $N$. 
It is a family of curves over the $\l$-line $L=\P^1-\{0,1,\infty\}$, and $X_{N,\l}$'s for positive integers $N$ form a tower of curves with respect to the division of $N$. 
In Section \ref{s2}, we study the Betti homology and cohomology of $X_{N,\l}$ by giving explicit bases, and express the complex periods in terms of solutions of the hypergeometric differential equations of order two satisfied by $F(a,b;1;\l)$ for $a, b \in N^{-1}\Z$. 
Consider the adelic homology group
\[H_1(X_{\infty,\l},\wh \Z) := \varprojlim_{N} H_1(X_{N,\l},\wh \Z).\]
Since $X_{N,\l}$ has a natural action of $\m_N \times \m_N$ where $\m_N$ is the group of $N$th roots of unity, $H_1(X_{\infty,\l}, \wh \Z)$ becomes a module over the 
completed group ring 
\[\L:=\wh\Z[[\wh\Z(1)\times\Z(1)]], \quad \wh\Z(1):=\varprojlim_N \m_N.\]  
We will show that the homology group is free of rank two over $\L$ (Corollary \ref{c4.2}). 

By Artin's comparison theorem, $H_1(X_{\infty,\l},\wh \Z)$ is identified with the \'etale homology group. 
Hence if $\l$ is an element of a number field $k$, it admits an action of the absolute Galois group $G_k:=\Gal(\ol\Q/k)$, and defines a cocycle 
\[M_\l \colon G_k \to \GL_2(\L). \] 
Our adelic hypergeometric function is defined to be its trace (see Section \ref{s3.3}). 
We remark that the adelic beta function is a cocycle $B\colon G_\Q \to \L^*$ defined by the Galois action on the adelic homology group of the Fermat curves, which is free and cyclic over $\L$ (see Section \ref{s3.4}). 
By evaluating $\tr M_\l$ at characters of $\wh\Z(1)\times \wh\Z(1)$, we define 
\[ F_\l(\s)(a, b) \in \wh \Z \ot \Z[\m_N]\] 
for $\s \in G_{k(\m_N)}$ and $a, b \in N^{-1}\Z/\Z$ (see Definition \ref{d3.9}), which are analogues of the complex functions $F(a,b;1;\l)$.  

If $\l$ is an element of a number field $k$ containing $\m_N$ and $X_{N,\l}$ has good reduction at a prime $v$ of $k$, then the Frobenius trace at $v$ 
on the eigencomponents of the $l$-adic \'etale cohomology of $X_{N,\l}$ with respect to the action of $\m_N \times \m_N$, where $l$ is a prime number such that $v \nmid l$,  
is computed by the Grothendieck-Lefschetz trace formula. By a point counting argument similarly as in Weil \cite{weil} for Fermat curves, the Frobenius traces are expressed by the hypergeometric functions $F(\a,\b;\e;\l \text{ mod } v)$ over the residue field $\k_v$ of $v$ (see Section \ref{s4.1} for the definition),  where $\a, \b$ are $N$-torsion multiplicative characters of $\k_v$ (see Theorem \ref{t4.8}). 
In this way, the adelic function $F_\l(\s)(a,b)$ interpolates all the hypergeometric functions of the same type over all finite fields, 
similarly as the adelic beta function interpolates all the Jacobi sums. 

Classically, many transformation formulas and summation formulas (formulas for special values) are known for complex Gaussian hypergeometric functions, and most of them have their finite analogues (see e.g. \cite{otsubo2}). 
While formulas for the adelic hypergeometric function imply their finite versions by the interpolation property,  
we can conversely derive adelic formulas from the finite versions by the Chebotarev density theorem. 
In Section \ref{s4.4}, we derive the adelic versions of the classical transformation formulas of Euler, Pfaff, Gauss and others and a  summation formula of Kummer (with the restriction to the case ``$c=1$") from the finite versions, 
although it would be more desirable to derive them as the \'etale realization of motivic versions (but see Remark \ref{r4.7}). 

In Section \ref{s5.1}, we give an alternative definition of our adelic hypergeometric function using the relative homology of open subvarieties of $X_{N,\l}$'s, which extends the range of interpolation. Finally in Section \ref{s5.2}, we propose a conjecture which enables us to define the adelic analogue of the generalized hypergeometric functions of type 
$\FFF{d+1}{d}{a_0,\dots, a_d}{1,\dots, 1}{\l}$ in a similar manner to the Gaussian case (i.e. $d=1$) given in this paper. 

There is another definition of $l$-adic Gaussian hypergeometric functions due to Furusho \cite{furusho}, which uses $l$-adic multiple zeta values. It would be interesting to see its relation to the $l$-adic component of our definition. 

\subsection*{Acknowledgements}
We would like to thank Hidekazu Furusho and Takahiro Tsushima for helpful discussions. 

\subsection*{Notations}
Let $\ol\Q$ be the algebraic closure of $\Q$ in $\C$. 
For a positive integer $N$, let $\mu_N$ denote the group of $N$th roots of unity in $\C$. 
Let $\Q^{\ab}=\bigcup_N \Q(\mu_N)$ be the maximal abelian extension of $\Q$, and 
$\Z^\ab=\bigcup_N \Z[\m_N]$ be its integer ring. 
Put $\wh\Z(1)=\varprojlim_N \mu_N$ where the transition map $\mu_M \to \mu_N$ when $N \mid M$ is the $(M/N)$th power map. 
Put $\be(t)=\exp(2\pi i t)$ and $\z_N=\be(1/N)$. 
Then $\z_\infty=(\z_N)_N$ is a topological generator of $\wh\Z(1)$. 
The gamma and the beta functions are defined by 
\[\G(s)=\int_0^\infty e^{-x}x^s\frac{dx}{x}, \quad B(s,t)=\int_0^1 x^s(1-x)^t\frac{dx}{x(1-x)}\]
if $\Re(s)>0$, $\Re(t)>0$. These satisfy functional equations
\[\G(s+1)=s\G(s), \quad \G(s)\G(1-s)=\frac{\pi}{\sin \pi s}, \quad B(s,t)=\frac{\G(s)\G(t)}{\G(s+t)}.\]

\section{Homology of hypergeometric curves}\label{s2}
\subsection{Hypergeometric curves}\label{s2.1}

Let $\P^1$ be the $\l$-line over $\C$ and put $L=\P^1-\{0,1,\infty\}$. 
For each positive integer $N$, let 
\[f_N \colon X_N \to L\] 
be a smooth projective curve defined in 
$\P^1\times \P^1 =\{([x_0:x_1],[y_0:y_1])\}$ over $L$ by  
\begin{equation}\label{e1.1}
(x_0^N-x_1^N)(y_0^N-y_1^N)=\l x_1^Ny_1^N.
\end{equation}
Its affine equation is 
\[(1-x^N)(1-y^N)=\l x^Ny^N,\]
where $x=x_1/x_0$, $y=y_1/y_0$. 
For $\l\in L$, let $X_{N,\l}=f_N^{-1}(\l)$ denote the fiber of $f_N$ over $\l$. 
It has $2N$ points at infinity: 
\[\{(x,\infty) \mid x^N=(1-\l)^{-1}\} \cup \{(\infty, y) \mid y^N=(1-\l)^{-1}\}.\]
The curves $X_N$ for various $N$ form a tower with the  transition morphisms when $N \mid M$
\[\pi_{M,N}\colon X_M \to X_N; \ (x,y)\mapsto (x^{M/N},y^{M/N}).\]  
The morphism $\pi_{N,1}$ is generically Galois with the Galois group 
\[\D_N:=\mu_N\times \mu_N,\]
which acts on $X_N$ as $(\z,\z').(x,y)=(\z x,\z' y)$. 
It ramifies at the $4N$ points on $X_N$ where $x, y \in \{0,\infty\}$ with the ramification indices $N$. 
Hence by the Riemann-Hurwitz formula noting that $X_1$ is rational, the genus of $X_N$ is $(N-1)^2$. 
The actions of $\D_M$ and $\D_N$ when $N\mid M$ are compatible under the morphism $\pi_{M,N}$ and the natural homomorphism $\D_M \to \D_N$. The character group $\wh{\D}_N=\Hom(\D_N,\C^\times)$ consists of  
\[\chi_N^{a,b}(\z,\z')=\z^a\z'^b, \quad (a,b)\in(\Z/N\Z)^{\oplus 2}.\]
Put elements of $\D_N$ as 
\[\x_N=(\z_N,1), \quad \y_N=(1,\z_N), \] 
so that $\D_N=\{\x_N^i\y_N^j \mid i,j \in \Z/N\Z\}$. 
Let $\Z[\D_N]$ be the group ring of $\D_N$ and 
$I_N \subset \Z[\D_N]$ be the ideal generated by the elements $\sum_{i\in\Z/N\Z} \x_N^i$ and $\sum_{i\in\Z/N\Z} \y_N^i$. 

\begin{rmk}\label{r1}
Let $C_{N,\l}^{a,b}$ be a smooth projective curve whose affine equation is 
\[v^N=(-u)^a (1-u)^{N-a}(1-\l u)^{N-b}.\]
For example, $C_{2,\l}^{1,1}$ is the Legendre elliptic curve. 
There exists a finite morphism $X_{N,\l}\to C_{N,\l}^{a,b}$ given by 
\[u=-\frac{x^N}{1-x^N}, \quad v=\frac{x^ay^b}{(1-x^N)y^N}\] 
One can express the motive (or the Jacobian variety) of $X_{N,\l}$ in terms of those of $C_{N,\l}^{a,b}$'s, 
but the former is easier to treat because it is a single curve for which we have no need of resolution of singularities.  
\end{rmk}

\subsection{Cohomology of $X_{N,\l}$}\label{s2.2}

For $a, b \in \{1,2,\dots, N-1\}$, define rational differential $1$-forms on $X_{N,\l}$ by 
\begin{align*}
&\o_{N,\l}^{a,b}=N\frac{x^ay^b}{1-x^N}\frac{dx}{x} =-N\frac{x^ay^b}{1-y^N}\frac{dy}{y}, 
\\&\y_{N,\l}^{a,b}=-\frac{b}{N\l}(1-y^N)\o_{N,\l}^{a,b}=-\frac{b}{\l}x^ay^b\frac{1-y^N}{1-x^N} \frac{dx}{x}
=\frac{b}{\l} x^ay^b \frac{dy}{y}.
\end{align*}
One easily verifies that $\o_{N,\l}^{a,b}$ are holomorphic and that  $\y_{N,\l}^{a,b}$ are non-holomorphic but of the second kind. 
These are eigenvectors for the $\D_N$-action with the character $\chi_N^{a,b}$, i.e. 
\[g^*\o_{N,\l}^{a,b}=\chi_N^{a,b}(g) \o_{N,\l}^{a,b}, \quad g^*\y_{N,\l}^{a,b}=\chi_N^{a,b}(g) \y_{N,\l}^{a,b}\]
for any $g\in \D_N$.  
These $1$-forms are normalized so that the compatibility 
\begin{equation}\label{e2.3}
\pi_{nN,N}^*(\o_{N,\l}^{a,b})=\o_{nN,\l}^{na,nb}, \quad \pi_{nN,N}^*(\y_{N,\l}^{a,b})=\y_{nN,\l}^{na,nb}
\end{equation}
holds. 
These $1$-forms define classes in the Betti cohomology  $H^1(X_{N,\l},\C)$, which by abuse of notation
we write by the same letters. 

\begin{rmk}With the notations in Remark \ref{r1}, 
\[\o_{N,\l}^{a,b}=\frac{v du}{u(1-u)(1-\l u)}, \quad \y_{N,\l}^{a,b}=\frac{b}{N} \frac{v du}{(1-u)(1-\l u)^2}.\]
\end{rmk}

We have the non-degenerate pairing 
\[H^1(X_{N,\l},\C)^{\ot 2} \os\cup\lra H^2(X_{N,\l},\C) \os\deg\simeq \C\] 
defined by the cup product and the degree map. 
We have the Hodge decomposition
\[H^1(X_{N,\l},\C) = H^{1,0}(X_{N,\l}) \oplus H^{0,1}(X_{N,\l}),\]
and the two components are dual to each other under the pairing as above. 

\begin{lem}\label{l2.1}
For $a, b, a', b' \in \{1,2,\dots, N-1\}$, we have:
\begin{enumerate}
\item $\deg(\o_{N,\l}^{a,b} \cup \o_{N,\l}^{a',b'})=0$, 
\item $\deg(\y_{N,\l}^{a,b} \cup \y_{N,\l}^{a',b'})=0$, 
\item $\deg(\o_{N,\l}^{a,b} \cup \y_{N,\l}^{a',b'})=\begin{cases} \frac{N^2}{\l(1-\l)} & \text{if $a+a'=b+b'=N$},\\
0 & \text{otherwise}.\end{cases}$
\end{enumerate}
In particular, $\o_{N,\l}^{a,b}$ and $\y_{N,\l}^{a,b}$ are non-trivial. 
\end{lem}

\begin{proof}
Since $\o_{N,\l}^{a,b}$ are holomorphic, (i) is immediate. 
Recall the formula: for $1$-forms of the second kind $\o$ and $\y$ 
\[\deg(\o \cup \y)=-\sum_{P\in X_{N,\l}} \Res_P\left(\left(\int\y\right) \o\right),\]
where $\int\y$ denotes a local primitive of $\y$ at $P$ and $\Res_P$ denotes the residue at $P$.   
The $1$-forms $\y_{N,\l}^{a,b}$ are holomorphic except at $P_i=(\z_N^i(1-\l)^{-1/N},\infty)$ ($i\in\Z/N\Z$),  
where $(1-\l)^{-1/N}$ is a fixed root. 
If we put $v=1/y$, then 
\begin{equation}\label{e2.4}
\o_{N,\l}^{a,b}=-N\frac{x^av^{N-b}}{1-v^N}\frac{dv}{v}, \quad \y_{N,\l}^{a',b'}=-\frac{b'}{\l}x^{a'}v^{-b'}\frac{dv}{v}.
\end{equation}
The statement (ii) follows immediately using the formula as above. 
Since 
\[\Res_{P_i}\left(\left(\int \y_N^{a',b'} \right) \o_N^{a,b} \right)=-\frac{N}{\l}\Res_{P_i} \left(\frac{x^{a+a'} v^{N-b-b'}}{1-v^N}\frac{dv}{v}\right)
\]
equals $-N \z_N^{a+a'}/{(\l(1-\l))}$ if $b+b'=N$ and $0$ otherwise, we obtain (iii)  (the latter case is also clear from the eigenvalues).  
\end{proof}

For a $\C$-vector space $V$ with $\D_N$-action (i.e. a $\C[\D_N]$-module), let 
\[V=\bigoplus_{(a,b)\in(\Z/N\Z)^{\oplus 2}} V(\chi_N^{a,b}),
\quad 
V(\chi_N^{a,b}):=\C\ot_{\C[\D_N],\chi^{a,b}} V,\]
be  the decomposition into the $\chi_N^{a,b}$-eigenspaces. 

\begin{ppn}\label{p2.2}For $a, b \in \{1,2,\dots, N-1\}$, we have 
\[H^{1,0}(X_{N,\l})(\chi_N^{a,b})=\C \o_{N,\l}^{a,b}, \quad H^{0,1}(X_{N,\l})(\chi_N^{a,b})=\C \y_{N,\l}^{a,b}.\] 
The space  $H^1(X_{N,\l},\C)$ (resp. $H^{1,0}(X_{N,\l})$, $H^{0,1}(X_{N,\l})$) has a $\C$-basis 
\[\{\o_{N,\l}^{a,b}, \y_{N,\l}^{a,b}\}_{1\le a,b \le N-1}\quad  \text{(resp.  $\{\o_{N,\l}^{a,b}\}_{1\le a,b \le N-1}$, \ $\{\y_{N,\l}^{a,b}\}_{1\le a,b \le N-1}$)}\] 
and is free of rank two (resp. one, one) over $(\Z[\D_N]/I_N)\ot\C$. 
\end{ppn}

\begin{proof}
Since $\o_{N,\l}^{a,b} \in H^{1,0}(X_{N,\l})=H^0(X_{N,\l}, \O^1_{X_{N,\l}})$, we have 
$\y_{N,\l}^{a,b} \in H^{0,1}(X_{N,\l})$ by Lemma \ref{l2.1}. 
Since $\o_{N,\l}^{a,b}$'s (resp. $\y_{N,\l}^{a,b}$'s) belong to different eigenspaces and 
$\dim_\C H^{1,0}(X_{N,\l})=(N-1)^2$, the statement on the basis follows. 
Since $I_N$ annihilates the $\chi_N^{a,b}$-part if $a\ne0$ and $b\ne 0$, $H^1(X_{N,\l},\C)$ and its Hodge components are 
$(\Z[\D_N]/I_N)\ot\C$-modules, and their freeness follows  
 since $\dim_\C (\Z[\D_N]/I_N)\ot\C =(N-1)^2$. 
\end{proof}

\subsection{Homology of $X_{N,\l}$}\label{s2.3}

We define elements of the Betti homology $\a_{N,\l}$, $\b_{N,\l} \in H_1(X_{N,\l},\Z)$ as follows. 
Fix a sufficiently small positive real number $\e$, and define loops $\a_{N,\e}$ and $\b_{N,\e}$ on $X_{N,\e}$ as follows. 
The loops $\a_{N,\l}$, $\b_{N,\l}$ for general $\l$ are defined to be the continuation of the loops at $\l=\e$ along a path in the $\l$-plane $\C-\{0,1\}$ (hence they depend on the choice of the path). 

First, define a loop $\a_{N,\e}\colon [0,1] \to X_{N,\e}$ by 
\[x(s)=1+\sqrt{\e} \,\be(-s), \quad |y(s)-1|<\sqrt{\e}.\] 
On the $x$-plane, this is a negatively oriented small circle around $1$. 
The cycle $\a_{N,\l}$ vanishes as $\l \to 0$ (in $\C-[1,\infty)$). 

Secondly, define a path $\d_{N,\e}\colon [0,1] \to X_{N,\e}$ from $(x,y)=(0,1)$ to $(1,0)$ by 
\[x(t)=\sqrt[N]{t}, \quad y(t)=\sqrt[N]{\frac{1-t}{1-(1-\e)t}} .\] 
Then
\[\b_{N,\e}:=((1-\x_N)(1-\y_N))_*\d_{N,\e}\]
becomes a cycle connecting four points $(x,y)=(0,1)$, $(1,0)$, $(0,\z_N)$ and $(\z_N,0)$. 

By abuse of notation,  we write the homology classes of $\a_{N,\l}$ and $\b_{N,\l}$ by the same letters. 
These classes are compatible under the transition maps, i.e. 
\[(\pi_{nN,N})_*(\a_{nN})=\a_N, \quad (\pi_{nN,N})_*(\b_{nN})=\b_{N}.\]
The group $\D_N$ acts on the homology and the cohomology of $X_{N,\l}$ respectively by the push-forward $g_*$ and the pull-back $g^*$ ($g \in \D_N$). Hence the group ring $\Z[\D_N]$ acts on the (co)homology, and we often omit writing the scripts ${}_*$ and ${}^*$ for these actions. 

\begin{thm}\label{t2.1}
The group $H_1(X_{N,\l},\Z)$ is a free $\Z[\D_N]/I_N$-module of rank two generated by $\a_{N,\l}$ and $\b_{N,\l}$. 
\end{thm}

\begin{proof}
By Proposition \ref{p2.2}, $H_1(X_{N,\l},\C)$ is annihilated by $I_N$ and $H_1(X_{N,\l},\Z)$, being torsion-free, is also annihilated by $I_N$.  
Since $\Z[\D_N]/I_N$ is a free abelian group generated by $\{\x_N^i\y_N^j \mid i,j=1,2,\dots, N-1\}$, 
it suffices to show that the set 
\[S=\left\{\x_N^i\y_N^j\a_{N,\l}, \x_N^i\y_N^j\b_{N,\l} \bigm| i,j=1,2,\dots, N-1\right\}\] 
of $2(N-1)^2$ elements 
is a $\Z$-basis of $H_1(X_{N,\l}\Z)$. 
To show this, we compute the intersection matrix for $S$. 
Since the intersection number is independent of $\l$, we are reduced to the case $\l=\e$. 
Since  the intersection number satisfies
\[\x_N^i\y_N^ju \cdot \x_N^{i'}\y_N^{j'}v=\x_N^{i-i'}\y_N^{j-j'}u \cdot v\]
for any $u, v \in H_1(X_{N,\l},\Z)$ and $i,j,i',j'$, we are reduced to the case $i'=j'=0$.  
It is easy to see that $\a_{N,\e}$ and $\b_{N,\e}$ meet at a unique point on $\d_{N,\e}$ where 
$s=1/2$ and $t=(1-\sqrt{\e})^N$, and 
$\a_{N,\e} \cdot \b_{N,\e}=-1$. 
More generally, $\x_N^i\y_N^j\a_{N,\e}$ is a vanishing cycle around $(x,y)=(\z_N^i,\z_N^j)$ and 
one sees that 
\begin{equation*}\label{int-n}
\x_N^i\y_N^j\a_{N,\e} \cdot \b_{N,\e}=\begin{cases} -1 & \text{if $(i,j)=(0,0)$ or $(1,1)$},\\ 0 & \text{otherwise}. \end{cases}
\end{equation*}
It is also easy to see that $\x_N^i\y_N^j\a_{N,\e} \cdot \a_{N,\e}=0$ for any $i, j$. 
Therefore, the intersection matrix with respect to $S$ (given the lexicographic order with respect to $(i,j)$ and 
setting $\x_N^i\y_N^j\a_{N,\l}$'s before $\x_N^i\y_N^j\b_{N,\l}$'s) is written as $\begin{bmatrix} O & -A \\ A & *\end{bmatrix}$ where 
$A=\begin{bmatrix} I & & &  \\ B&I \\& \ddots & \ddots  \\ & & B&I\end{bmatrix}$ is a block matrix of size $N-1$, 
where $I$ is the identity matrix and 
$B=\begin{bmatrix} 0 \\ 1& 0 \\ &\ddots & \ddots \\ & & 1& 0\end{bmatrix}$ both of size $N-1$.
Since $\det\begin{bmatrix} O & -A \\ A & *\end{bmatrix}=(\det A)^2=1$, the theorem follows. 
\end{proof}

\begin{rmk}
The intersection numbers of $\x_N^i\y_N^j\b_{N,\l}$'s, written $*$ above, can be computed using \cite[Corollary 4.3]{otsubo1}. 
\end{rmk}

\subsection{Complex hypergeometric functions}\label{s2.4}

Let us recall basic facts about the classical Gaussian hypergeometric functions (cf. \cite{erdelyi}). 
For complex parameters $a$, $b$, $c$ with $c \not\in\Z_{\le 0}$, the hypergeometric function is defined by the power series 
\[F(a,b;c;\l)=\sum_{n=0}^\infty \frac{(a)_n(b)_n}{(1)_n(c)_n}\l^n,\]
convergent on $|\l|<1$, 
where the Pochhammer symbol is defined by 
\[(a)_n=\frac{\G(a+n)}{\G(a)}=\prod_{i=0}^{n-1}(a+i).\] 
It is a solution, holomorphic at $\l=0$, of the hypergeometric differential equation
\begin{equation}\label{hgf-de}
\left[\frac{d^2}{d\l^2}+\left(\frac{c}{\l}-\frac{a+b+1-c}{1-\l}\right)\frac{d}{d\l} -\frac{ab}{\l(1-\l)}\right]y=0.
\end{equation}
By the symmetry with respect to $\l \leftrightarrow 1-\l$, 
$F(a,b;a+b+1-c;1-\l)$ 
is another solution of \eqref{hgf-de}, holomorphic at $\l=1$. 
In particular, the two functions are continued to multivalued functions on $L$. 
We have the integral representation of Euler type
\begin{equation}\label{hgf-i}
B(b,c-b)F(a,b;c;\l)=\int_0^1 (1-\l t)^{-a} t^b(1-t)^{c-b} \frac{dt}{t(1-t)}
\end{equation}
if $0<\Re(b)<\Re(c)$, from which follows the Euler-Gauss summation formula
\begin{equation}\label{euler-gauss}
F(a,b;c;1)=\frac{\G(c)\G(c-a-b)}{\G(c-a)\G(c-b)}
\end{equation}
if $\Re(c-a-b)>0$. 
The derivative of $F(a,b;c;\l)$ is again a hypergeometric function with integrally shifted parameters: 
\begin{equation}\label{hgf-d}
\frac{d}{d\l} F(a,b;c;\l)=\frac{ab}{c} F(a+1,b+1;c+1;\l).
\end{equation}

\subsection{Periods of $X_{N,\l}$}

Let $a, b \in \{1,2,\dots, N-1\}$ and put for brevity
\begin{align*}
& f_N^{a,b}(\l)=2\pi i F\left(\frac{a}{N},\frac{b}{N};1;\l\right), 
\\& g_N^{a,b}(\l)=B\left(\frac{a}{N},\frac{b}{N}\right)F\left(\frac{a}{N},\frac{b}{N};\frac{a}{N}+\frac{b}{N};1-\l\right).
\end{align*}
These form a fundamental solution set of the differential equation \eqref{hgf-de} with the parameters $(a/N, b/N, 1)$ in place of $(a, b, c)$. 

\begin{thm}\label{t2.2}
For $a, b\in\{1,2,\dots, N-1\}$, we have 
\begin{equation*}
\begin{bmatrix} 
\int_{\a_{N,\l}} \o^{a,b}_{N,\l} & \int_{\b_{N,\l}} \o^{a,b}_{N,\l}  \\
\int_{\a_{N,\l}} \y^{a,b}_{N,\l}&\int_{\b_{N,\l}} \y^{a,b}_{N,\l}
\end{bmatrix}
=
\begin{bmatrix} f_N^{a,b}(\l) & (1-\z_N^a)(1-\z_N^b)g_N^{a,b}(\l) \\
\frac{d}{d\l} f_N^{a,b}(\l)  & (1-\z_N^a)(1-\z_N^b)\frac{d}{d\l} g_N^{a,b}(\l)
\end{bmatrix}
\end{equation*}
where the entries are multivalued functions on $L$. 
\end{thm}

\begin{proof}
To prove the equality of the first columns, we can assume that $|\l|$ is sufficiently small. 
Since $y^{-N}=1-\frac{\l x^N}{x^N-1}$, we have by the binomial theorem 
\begin{align*}
y^b=\sum_{n=0}^\infty \frac{(\frac{b}{N})_n}{(1)_n} \left(\frac{\l x^N}{x^N-1}\right)^n .
\end{align*}
Hence, letting $u=x^N$, 
\begin{align*}
\int_{\a_{N,\l}} \o^{a,b}_{N,\l} 
= \sum_{n=0}^\infty \frac{(\frac{b}{N})_n}{(1)_n} \l^n \oint \frac{u^{\frac{a}{N}+n-1}}{(u-1)^{n+1}}\, du,
\end{align*}
where the integral is taken along a positively oriented small circle around $u=1$. 
Since 
\[u^{\frac{a}{N}+n-1}=\sum_{i=0}^\infty \binom{\frac{a}{N}+n-1}{i}(u-1)^i,\] 
and $\binom{\frac{a}{N}+n-1}{n}=(\frac{a}{N})_n/(1)_n$,  
the formula for $\int_{\a_{N,\l}} \o^{a,b}_{N,\l}$ follows by the residue theorem. 
The formula for $\int_{\a_{N,\l}} \y^{a,b}_{N,\l}$ follows similarly using \eqref{hgf-d}. 

To prove the equality of the second columns, we can assume that $|1-\l|$ is sufficiently small. 
Then we have 
\begin{align*}
\int_{\b_{N,\l}} \o^{a,b}_{N,\l}
& =\int_{\d_{N,\l}} ((1-\x_N)(1-\y_N))^*\o^{a,b}_{N,\l}
=(1-\z_N^a)(1-\z_N^b)\int_{\d_{N,\l}} \o^{a,b}_{N,\l}\\
&=(1-\z_N^a)(1-\z_N^b)\int_0^1 (1-(1-\l)t)^{-\frac{b}{N}} t^\frac{a}{N} (1-t)^\frac{b}{N} \frac{dt}{t(1-t)}.
\end{align*}
Hence the formula for $\int_{\b_{N,\l}} \o^{a,b}_{N,\l}$
follows by \eqref{hgf-i}. 
The formula for $\int_{\b_{N,\l}} \y^{a,b}_{N,\l}$ follows similarly using \eqref{hgf-d} and $B(\frac{a}{N}+1,\frac{b}{N})=\frac{a}{a+b}B(\frac{a}{N},\frac{b}{N})$. 
\end{proof}

\begin{rmk}
In \cite{asakura}, a higher dimensional analogue of $X_{N,\l}$ is considered and the periods along an analogue of   $\a_{N,\l}$ are computed. 
\end{rmk}

\begin{cor}\label{c2.9}
Let $\iota$ be an involution of $X_{N,\l}$ defined by $\iota(x,y)=(y,x)$. 
Then, $\iota_*$ acts on $H_1(X_{N,\l},\Z)$ as multiplication by $-1$. 
\end{cor}

\begin{proof}
Since $f_N^{a,b}=f_N^{b,a}$ and  $g_N^{a,b}=g_N^{b,a}$, it suffices by 
Proposition \ref{p2.2} and Theorem \ref{t2.2} to show that $\iota^*\o_{N,\l}^{a,b}=-\o_{N,\l}^{b,a}$ and $\iota^*\y_{N,\l}^{a,b}=-\y_{N,\l}^{b,a}$ in cohomology.  
The former is evident, and since
\[\iota^*\y_{N,\l}^{a,b}+\y_{N,\l}^{b,a}=\frac{b}{\l} x^by^a\frac{dx}{x}+ \frac{a}{\l} x^by^a \frac{dy}{y}=d\left(\frac{x^by^a}{\l}\right)\]
as $1$-forms, the latter follows. 
\end{proof}

\subsection{Local monodromy}\label{s2.6}

For $s\in\{0,1,\infty\}$, 
let $\mathscr{D}_s\subset \P^1$ be a small open disk centered at $s$ and put $\mathscr{D}_s^*=\mathscr{D}_s-\{s\}$. 
Take a base point $\l \in \mathscr{D}^*$ and let
\[T_s\colon H_1(X_{N,\l},\Z) \to H_1(X_{N,\l},\Z)\] 
be the local monodromy operator at $s$. 
It commutes with the $\D_N$-action.

\begin{ppn}\label{p2.7}\ 
\begin{enumerate}
\item When $s=0$, we have 
\[(1-T_0)(\a_{N,\l})=0, \quad  (1-T_0)(\b_{N,\l})=(1-\x_N)(1-\y_N) \a_{N,\l}.\]
\item  When $s=1$, we have 
\[(1-T_1)(\b_{N,\l})=0, \quad (1-T_1)(\a_{N,\l}) = -\x_N^{-1}\y_N^{-1} \g_{N,\l},\]
where
\[\g_{N,\l}:=(1-\x_N\y_N)\a_{N,\l}+\b_{N,\l}.\] 
\item When $s=\infty$, we have 
\begin{align*}
&(1-T_\infty)(\a_{N,\l})=(2-\x_N-\y_N) a_{N,\l}+\b_{N,\l}, 
\\&(1-T_\infty)(\b_{N,\l})=-(1-\x_N)(1-\y_N)\a_{N,\l}.
\end{align*}
\end{enumerate}
\end{ppn}

\begin{proof}
It suffices to compare the periods using Theorem \ref{t2.2}, and we are reduced to the monodromy of hypergeometric functions. 
Since $T_s$ and $\frac{d}{d\l}$ commutes, it suffices to compare the periods of $\o_N^{a,b}$. 

Since its period along $\a_{N,\l}$ (resp $\b_{N,\l}$) is given by the functions $f_N^{a,b}$ (resp. $g_N^{a,b}$) which are holomorphic at $\l=0$ (resp. at $\l=1$), the first formula of (i) (resp. of (ii)) follows. 

To prove the second formula of (i), put 
\begin{align*}
& h_N^{a,b}(\l)= \frac{2\pi i}{B(\frac{a}{N},\frac{b}{N})}\sum_{n=0}^\infty \frac{(\frac{a}{N})_n(\frac{b}{N})_n}{(1)_n(1)_n} k_n \l^n, \label{e2.5}
\\& k_n:=2\psi(1+n)-\psi\left(\frac{a}{N}+n\right)-\psi\left(\frac{b}{N}+n\right),
\end{align*}
where $\psi(s)=\G'(s)/\G(s)$ is the digamma function. 
Then we have by \cite[2.3. (2)]{erdelyi} 
\[g_N^{a,b}(\l)=h_N^{a,b}(\l) - \frac{1}{2\pi i}\log\l \cdot f_N^{a,b}(\l).\]
Since $h_N^{a,b}$ is holomorphic at $\l=0$, 
\[(1-T_0)g_N^{a,b}(\l)= f_N^{a,b}(\l).\]
Noting 
\[\int_{(1-\x_N)(1-\y_N)\a_{N,\l}} \o_{N,\l}^{a,b}=(1-\x_N^a)(1-\y_N^b)\int_{\a_{N,\l}} \o_{N,\l}^{a,b},\]
we obtain the second formula of (i). 

To prove the second formula of (ii), suppose first that $a+b \ne N$. Then we have by \cite[2.9. (33)]{erdelyi}
\[f_N^{a,b}=\frac{(1-\z_N^{-a})(1-\z_N^{-b})}{1-\z_N^{-a-b}} g_N^{a,b} + k_N^{a,b}\] 
for a function $k_N^{a,b}$ of local exponent $1-\frac{a}{N}-\frac{b}{N}$ at $\l=1$. 
In fact, 
\begin{align*}
k_N^{a,b}(\l)=&2\pi i \frac{\G(\frac{a}{N}+\frac{b}{N}-1)}{\G(\frac{a}{N})\G(\frac{b}{N})} \times \\&
\quad (1-\l)^{1-\frac{a}{N}-\frac{b}{N}}F\left(1-\frac{a}{N},1-\frac{b}{N};2-\frac{a}{N}-\frac{b}{N};1-\l\right).
\end{align*}
It follows that
\begin{equation}\label{e2.1}
(1-T_1)f_N^{a,b}=-\z_N^{-a-b}\left((1-\z_N^{a+b}) f_N^{a,b}+(1-\z_N^a)(1-\z_N^b) g_N^{a,b}\right).
\end{equation}
Secondly, suppose that $a+b=N$. Then we have by \cite[2.3. (2)]{erdelyi} 
\[f_N^{a,b}(\l)=h_N^{a,b}(1-\l)+\frac{1}{2\pi i} (1-\z_N^{-a})(1-\z_N^{-b})\log(1-\l)g_N^{a,b}(\l),\]
where $h_N^{a,b}$ is as above. 
It follows that 
\[(1-T_1)f_N^{a,b}=-(1-\z_N^{-a})(1-\z_N^{-b}) g_N^{a,b}.\]
Therefore \eqref{e2.1} holds for any $a$ and $b$, and the second formula of (ii) follows. 

Finally, (iii) follows easily from (i) and (ii) using $T_\infty=(T_1T_0)^{-1}$. 
\end{proof}

\begin{rmk}
From Proposition \ref{p2.7}, one verifies the 
unipotency $(1-T_0)^2=0$, the quasi-unipotency $(1-T_1^N)^2=0$, and the finiteness $T_\infty^N=1$. 
\end{rmk}

\begin{cor}\label{t2.3}\ 
\begin{enumerate}
\item
The groups $\Ker(1-T_0)$ and $\Im(1-T_0)$ are free cyclic $\Z[\D_N]/I_N$-modules generated respectively by 
$\a_{N,\l}$ and $(1-\x_N)(1-\y_N)\a_{N,\l}$. 
\item
The groups $\Ker(1-T_1)$, $\Im(1-T_1)$ and $\Coker(1-T_1)$ are free cyclic $\Z[\D_N]/I_N$-modules generated respectively by 
$\b_{N,\l}$, $\g_{N,\l}$ and the class of $\a_{N,\l}$. 
\item The map $1-T_\infty$ is an injection with finite cokernel. 
\end{enumerate}
\end{cor}

\begin{proof}This is immediate from Theorem \ref{t2.1} and Proposition \ref{p2.7}. 
\end{proof}

\subsection{Specializations}\label{s2.7}

To consider the specializations at $s \in \{0,1,\infty\}$, we extend the morphism $f_N$ to a proper flat morphism $\wt f_N\colon \wt{X}_N \to \P^1$, 
where the surface $\wt{X}_N$ is regular and any fiber $X_{N,s}:=\wt{f}_N^{-1}(s)$, $s \in\{0,1,\infty\}$ is a divisor with normal crossings.  
Let $f_N'\colon X'_N \to \P^1$ be defined by the same equation \eqref{e1.1} as $f_N$. Its fiber over $\l=1$ has a unique singularity $(x,y)=(\infty,\infty)$ unless $N=1$. 
Blowing up the surface $X'_N$ at this point, we obtain $\wt{X}_N$, to which the morphism $f_N$ and the $\D_N$-action extend. 
This construction is compatible for all $N$, i.e. the morphism $\pi_{M,N}$ (see Section \ref{s2.1}) extends to a morphism $\wt{X}_M\to \wt{X}_N$ over $\P^1$. 

The three special fibers are described as follows.

\begin{itemize}
\item
The fiber $X_{N,0}$ is a union of $2N$ rational curves 
\[L_i : x=\z_N^i, \quad L'_j : y=\z_N^j \quad (i,j\in\Z/N\Z),\] 
and $L_i$ and $L_j'$ intersect each other transversally at $P_{i,j}:=(\z_N^i,\z_N^j)$.  
\item 
The fiber
 $X_{N,1}$ is a union of the Fermat curve $C_N$ defined by 
\[x_0^N+y_0^N=z_0^N\]
and a rational curve $E$ with multiplicity $N$, intersecting each other transversally at the $N$ points 
\[Q_k:=(\z_N^k\z_{2N}:1:0)\quad (k\in\Z/N\Z).\] 
The action of $\D_N$ respects $C_N$ and $E$, and it acts on $C_N \cap E$ through the homomorphism 
$\chi_N^{1,-1}\colon \D_N \to \mu_N$ (see Section \ref{s2.1}). 
\item 
The fiber $X_{N,\infty}$ is a union of two rational curves both with multiplicity $N$, intersecting each other transversally at one point. 
\end{itemize}
By the  description of $X_{N,\infty}$ as above, we have 
\[H_1(X_{N,\infty},\Z)=0\] 
(this also follows from Corollary \ref{t2.3} (iii) and Lemma \ref{l2.12} below).  
The structure of $H_1(X_{N,s},\Z)$ for $s=0, 1$ will be given below. 

For $s\in\{0,1,\infty\}$, let $\sD_s$, $\sD_s^*$ and $\l \in \sD_s^*$ be as in Section \ref{s2.6}, and put  $\sX_{N,s}=\wt{f}_N^{-1}(\sD_s)$. 
Then the specialization map is defined by the composite 
\begin{equation*}\label{sp-betti}
\sp_s\colon H_1(X_{N,\l},\Z) \to H_1(\sX_{N,s},\Z) \os\simeq\lra H_1(X_{N,s},\Z),
\end{equation*}
where the second map is the inverse of the push-forward.  

\begin{lem}\label{l2.12}
The map $\sp_s$ factors through $\Coker(1-T_s)$ and is surjective. 
In particular, it is a homomorphism of $\Z[\D_N]/I_N$-modules.  
\end{lem}

\begin{proof}
Put $\sX_{N,s}^*=f_N^{-1}(\sD_s^*)=\sX_{N,s} - X_{N,s}$. 
First, we have the localization exact sequence 
\[H_1(\sX_{N,s}^*,\Z) \to H_1(\sX_{N,s},\Z) \to H_1(\sX_{N,s},\sX_{N,s}^*;\Z) \to 0,\]
but since 
$H_1(\sX_{N,s},\sX_{N,s}^*,\Z) \simeq H^{3}(X_{N,s},\Z)=0$ 
by the Poincar\'e-Lefschetz duality, the first map is surjective. 
Secondly, we have Wang's exact sequence (cf. \cite[III, 6]{serre3})
\[H_1(X_{N,\l},\Z) \xrightarrow{1-T_s} H_1(X_{N,\l},\Z)  \to H_1(\sX_{N,s}^*,\Z)  \xrightarrow{\pd} H_0(X_{N,\l},\Z) \to 0,\]
and the factorization follows. 
This is a sequence of $\Z[\D_N]/I_N$-modules since the $\D_N$-action commutes with $T_s$ and $\D_N$ acts trivially on $H_0(X_{N,\l},\Z)\simeq \Z$. 
By the functoriality of Wang's sequence, we have the commutative diagram
\[\xymatrix{
H_1(\sX_{N,s}^*,\Z)  \ar[r]^\pd  \ar[d]& H_0(X_{N,\l},\Z) \ar[d]_\simeq \\ 
H_1(\sD_s^*,\Z) \ar[r]^\simeq & H_0(\{\l\},\Z)
}
\]
where we used $H_1(\{\l\},\Z)=0$ for the horizontal isomorphism. 
If $i\colon \sD_s \to \sX_{N,s}$ is a section, given for example by $(x,y)=(0,1)$, then the composition 
\[H_1(\sD_s^*,\Z) \xrightarrow{i_*} H_1(\sX_{N,s}^*,\Z) \xrightarrow{\pd} H_0(X_{N,\l},\Z)\]
is an isomorphism, and hence $H_1(\sX_{N,s}^*,\Z)$ is generated by $\Coker(1-T_s)$ and $\Im(i_*)$. 
Since $i(\sD_s)$ is contractible, $\Im(i_*)$ maps trivially to $H_1(\sX_{N,s},\Z)$ and we obtain the desired surjectivity. 
\end{proof}

\begin{thm}\label{t2.5} \ 
\begin{enumerate}
\item There is an isomorphism of $\Z[\D_N]/I_N$-modules 
\[H_1(X_{N,0},\Z) \simeq \Z[\D_N](1-\x_N)(1-\y_N) \subset \Z[\D_N], \]
\item The group $H_1(X_{N,0},\Z)$ is a free cyclic $\Z[\D_N]/I_N$-module generated by $\sp_0(\b_{N,\l})$. 
\item The group $\Ker(\sp_0)$ is a free cyclic $\Z[\D_N]/I_N$-module generated by $\a_{N,\l}$. 
\end{enumerate}
\end{thm}

\begin{proof}
(i) We have an exact sequence 
\[0 \to H_1(X_{N,0},\Z) \to \bigoplus_{i,j} H_0(P_{i,j},\Z) \to \bigoplus_i H_0(L_i,\Z) \oplus 
\bigoplus_j H_0(L_j',\Z).\]
Identifying the fundamental class of $H_0(P_{i,j},\Z)$ with the element $\x_N^i\y_N^j \in \D_N$, 
we have an isomorphism of $\Z[\D_N]$-modules 
\[H_1(X_{N,0},\Z)  \simeq \Ker (\Z[\D_N] \xrightarrow{\t} \Z^{\oplus 2N}),\]
where
\[\t\Bigl( \sum_{i,j} n_{i,j}\x_N^i\y_N^j \Bigr)= \left(\Bigl(\sum_j n_{i,j}\Bigr)_i, \Bigl(\sum_i n_{i,j}\Bigr)_j \right).\] 
It is easy to see that $\Ker(\t)=\Z[\D_N](1-\x_N)(1-\y_N)$.  

(ii) Since the annihilator of $\Z[\D_N](1-\x_N)(1-\y_N)$ in $\Z[\D_N]$ is $I_N$, the free cyclicity follows. 
One sees that the path $\d_{N,\l}$ specializes to a segment from $(x,y)=(0,1)$ to $(1,1)$ on $L_0'$ followed by a segment from $(1,1)$ to $(1,0)$ on $L_0$. 
Hence the cycle $\b_{N,\l}$ specializes to the class of a loop connecting 
$P_{0,0}$, $P_{0,1}$, $P_{1,1}$, $P_{1,0}$ and $P_{0,0}$ respectively 
on $L_0$, $L_1'$, $L_1$ and $L_0'$, which agrees $\pm (1-\x_N)(1-\y_N)$ under the isomorphism of (i). 

(iii) By Proposition \ref{p2.7} (i) and Lemma \ref{l2.12}, 
\[(1-\x_N)(1-\y_N)\sp_0(\a_{N,\l})=\sp_0\circ (1-T_0)(\b_{N,\l})=0.\] 
Since $(1-\x_N)(1-\y_N)$ is not a zero-divisor of $\Z[\D_N]/I_N$, we have $\sp_0(\a_{N,\l})=0$ by (ii). 
Hence the statement follows by (ii) and Theorem \ref{t2.1}. 
\end{proof}

\begin{thm}\label{t2.4}\ 
\begin{enumerate} 
\item There is an exact sequence of $\Z[\D_N]/I_N$-modules 
\[0 \to H_1(C_N,\Z) \to H_1(X_{N,1},\Z) \ra \Ker(\Z[\mu_N] \xrightarrow{\deg} \Z) \to 0,\]
where $\D_N$ acts on $\m_N$ (hence on the free abelian group $\Z[\m_N]$) via $\chi_N^{1,-1}$. 
\item The group $H_1(X_{N,1},\Z)$ (resp. $H_1(C_N,\Z)$) is a free cyclic (resp. a cyclic) $\Z[\D_N]/I_N$-module generated by $\sp_1(\a_{N,\l})$ (resp. by $\sp_1(\b_{N,\l})=-(1-\x_N\y_N) \sp_1(\a_{N,\l})$).  
\item The group $\Ker(\sp_1)$ is a free cyclic $\Z[\D_N]/I_N$-module generated by $\g_{N,\l}$ (see Proposition \ref{p2.7} (ii)). 
\end{enumerate}
\end{thm}

\begin{proof}
(i) 
We have an exact sequence
\[0\to H_1(C_N,\Z) \to H_1(X_{N,1},\Z) \to \bigoplus_k H_0(Q_k,\Z) \to H_0(C_N,\Z) \oplus H_0(E,\Z).\]
Identifying the fundamental class of $H_0(Q_k,\Z)$ with $\z_N^k \in \mu_N$, we obtain the desired sequence. 
 
 (ii) By Corollary \ref{t2.3} (ii) and Lemma \ref{l2.12}, we have $\sp_1(\g_{N,\l})=0$ and $\sp_1(\a_{N,\l})$ generates 
 $H_1(X_{N,1},\Z)$ over $\Z[\D_N]/I_N$. 
Since the genus of $C_N$ is $(N-1)(N-2)/2$, we have by (i) 
\[\rank H_1(X_{N,1},\Z)=(N-1)^2=\rank \Z[\D_N]/I_N,\] 
hence the freeness of $H_1(X_{N,1},\Z)$ follows. 
On the other hand, by Rohrlich \cite{rohrlich} (cf. \cite[Corollary 4.3]{otsubo1}),  
$H_1(C_N,\Z)$ is a cyclic $\Z[\D_N]/I_N$-module generated by $\sp_1(\b_{N,\l})$, which agrees $-(1-\x_N\y_N) \sp_1(\a_{N,\l})$ by Corollary \ref{t2.3} (ii). 

(iii) Since $H_1(X_{N,\l},\Z)$ is freely generated by $\a_{N,\l}$ and $\g_{N,\l}$ by Theorem \ref{t2.1}, 
the assertion follows. 
\end{proof}

\begin{rmk}
The annihilator of $H_1(C_N,\Z)$ in $\Z[\D_N]/I_N$ is generated by the images of $\sum_{i\in\Z/N\Z} (\x_N\y_N)^i$ and $\sum_{0\le i<j<N} \x_N^i\y_N^j$ (see \cite[Proposition 4.4]{otsubo1}). 
\end{rmk}

\section{The adelic hypergeometric function}

\subsection{$\L$-modules}\label{s3.1}

Define a profinite group as 
\[\D_\infty:=\varprojlim_N \D_N.\] 
It is isomorphic to $\wh\Z(1)\times \wh\Z(1)$ and generated topologically by  $\x:=(\x_N)$ and $\y:=(\y_N)$.  
We put the completed group ring with $\wh\Z$-coefficients as 
\[\L:=\wh\Z[[\D_\infty]] = \varprojlim_N \wh\Z[\D_N].\]
Elements of $\L$ is regarded as  $\wh\Z$-valued measures on $\D_\infty$: the measure associated to an element 
$\Phi=(\sum_{g_N\in \D_N} c_N(g_N) g_N)_N \in \L$ is given by 
\[\mu_\Phi(g + N\D_\infty)=c_N(g \text{ mod } N)\]
for any $g \in \D_\infty$ and $N$. 

\begin{rmk}\label{r3.1}
If we consider the pro-$l$ part of $\L$ for a prime number $l$, we have an isomorphism
\[\Z_l[[\D_{l^\infty}]]:=\varprojlim_n \Z_l[[\D_{l^n}]] \os\simeq\lra \Z_l[[S,T]]; \quad (\x_{l^n}) \mapsto 1-S, \ (\y_{l^n}) \mapsto 1-T\]
to the ring of two-variable formal power series with coefficients in $\Z_l$. 
\end{rmk}

The adelic homology group
\[H_1(X_{N,\l},\wh \Z) := \varprojlim_n H_1(X_{N,\l},\Z/n\Z)\]
is a $\wh\Z[\D_N]$-module, hence is a $\L$-module. 
Since the transition maps $(\pi_{M,N})_*$ are compatible with the $\L$-actions, 
\[H_1(X_{\infty,\l},\wh\Z):= \varprojlim_N H_1(X_{N,\l},\wh\Z)\]
becomes a  $\L$-module, continuous with respect to the profinite topology. 
The compatible systems of homology classes $(\a_{N,\l})$, $(\b_{N,\l})$ and $(\g_{N,\l})$ (see Section \ref{s2.3} and Proposition \ref{p2.7} (ii)) define respectively 
\[\a_\l, \ \b_\l, \ \g_\l=(1-\x\y)_*\a_\l+\b_\l \ \in H_1(X_{\infty,\l},\wh\Z).\] 
Define similarly other $\L$-modules such as $H_1(X_{\infty,s},\wh\Z)$ for $s\in\{0,1,\infty\}$ or $H_1(C_{\infty},\wh\Z)$ (see Section \ref{s2.7}). 
Then we have $\L$-homomorphisms 
\[H_1(X_{\infty,\l},\wh\Z) \xrightarrow{\sp_s} H_1(X_{\infty,s},\wh\Z), \quad 
 H_1(C_{\infty},\wh\Z)\hookrightarrow H_1(X_{\infty,1},\wh\Z). \]

\begin{rmk}
Put 
\[H_1(X_{\infty,\l},\Z)= \varprojlim_N H_1(X_{N,\l},\Z), \quad H^1(X_{\infty,\l},\C)= \varinjlim_N H^1(X_{N,\l},\C).\]
Then $\a_\l$, $\b_\l$ and $\g_\l$ are contained in the first group. For each $(s,t) \in (\Q/\Z)^{\oplus 2}$, $s, t \ne 0$, 
elements $\o_\l(s,t)$ and $\y_\l(s,t)$ of the latter group are defined to be the classes respectively of 
$\o_{N,\l}^{a,b}$ and $\y_{\l,N}^{a,b}$ where $(s,t)=(a/N,b/N)$ (see \eqref{e2.3} for the independence on the presentation). 
We have the non-degenerate period paring 
\[H_1(X_{\infty,\l},\Z) \times H^1(X_{\infty,\l},\C) \to \C,\]
written by $\int$. By Theorem \ref{t2.2}, the elements $\a_\l$ and $\b_\l$ are characterized by 
\[\bmat{\int_{\a_\l} \o_\l(s,t)& \int_{\b_\l} \o_\l(s,t) \\ \int_{\a_\l} \y_\l(s,t)& \int_{\b_\l} \y_\l(s,t)} 
= \bmat{f(s,t;\l) & (1-\be(s))(1-\be(t))g(s,t;\l) \\ \frac{d}{d\l}f(s,t;\l) & (1-\be(s))(1-\be(t))\frac{d}{d\l} g(s,t;\l)}\]
for any $(s,t)$ as above, where we put
\[f(s,t;\l)=2\pi i F(s,t;1;\l),  \quad g(s,t;\l)= B(s,t)F(s,t;s+t;1-\l). \]
\end{rmk}

For the freeness over $\L$, we will use the following. 
\begin{lem}\label{l3.1}
The natural map 
\[\wh\Z[[\D_\infty]] \to \varprojlim_{N,n} \Z/n\Z[\D_N]/I_N\] 
is an isomorphism of topological rings. 
\end{lem}

\begin{proof}For a fixed $n$, consider the exact sequence of inverse systems with respect to $N$
\[0 \to I_N/n \to \Z/n\Z[\D_N] \to \Z/n\Z[\D_N]/I_N \to 0. \]
Since the map $\Z[\D_{mN}]\to\Z[\D_N]$ sends $I_{mN}$ onto $mI_N$, 
the map $I_{mN}/n \to I_N/n$ is trivial whenever $n\mid m$. 
In particular, the system $(I_N/n)_N$ satisfies the Mittag-Leffler condition. 
Hence $\varprojlim_N I_N/n = \varprojlim_N^1 I_N/n=0$ and $\Z/n\Z[[\D_\infty]] = \varprojlim_{N} \Z/n\Z[\D_N]/I_N$ for any $n$. 
Taking $\varprojlim_n$, we obtain the lemma. 
\end{proof}

By Lemma \ref{l3.1}, we obtain the following immediate consequences of Theorems \ref{t2.1}, \ref{t2.5}, \ref{t2.4}, and the freeness of $H_1(C_{\infty},\wh\Z)$ \cite[Theorem 6.1]{otsubo1}.
\begin{cor}\label{c4.2} \ 
\begin{enumerate}
\item $H_1(X_{\infty,\l},\wh\Z)$ is a free $\L$-module of rank two generated by $\a_\l$ and $\b_\l$. 
\item $\Ker(\sp_0)$ and $H_1(X_{\infty,0},\wh\Z)$ are free cyclic $\L$-modules generated respectively by $\a_{\l}$ and $\sp_0(\b_{\l})$. 
\item $\Ker(\sp_1)$, $H_1(X_{\infty,1},\wh\Z)$ and  $H_1(C_{\infty},\wh\Z)$ are free cyclic $\L$-modules generated respectively by $\g_\l$, $\sp_1(\a_\l)$ and $\sp_1(\b_\l)=-(1-\x\y)\sp_1(\a_\l)$.
\end{enumerate}
\end{cor}

\subsection{Galois action}\label{s3.2}
A general reference for this section is \cite{4demi}.  
For a smooth projective curve $X$ over $\C$, define the \'etale homology with $\Z/n\Z$-coefficients by 
\[H^\et_1(X,\Z/n\Z) = H^1_\et(X, \Z/n\Z(1)).\]
We have the perfect Poincar\'e duality pairing
\begin{equation}\label{e4.1}
\langle \ , \ \rangle \colon  H_1^\et(X,\Z/n\Z) \ot H_\et^1(X,\Z/n\Z) \to \Z/n\Z.
\end{equation}
If a group $\D$ acts on $X$ from the left, it acts on the cohomology from the right by the pull-back, written $g^*$ for $g\in \D$. It acts on the homology from the left by $g_*:=(g^{-1})^*$. Then the actions are adjoint to each other, i.e. $\langle g_*\g ,\o \rangle = \langle \g, g^* \o\rangle$. 
Suppose that $X$ is defined over a subfield $k \subset \ol\Q$, and let 
\[G_k=\mathrm{Gal}(\ol\Q/k)\]
be the absolute Galois group. 
If $X_k$ is a model of $X$ over $k$, we have for any $r\in\Z$ 
\[H_\et^1(X_k \ot_k \ol\Q, \Z/n\Z(r)) \simeq H_\et^1(X,\Z/n\Z(r)).\]
Hence $G_k$ acts on the homology and the cohomology from the left in such a way that   
$\langle \s\g, \s\o \rangle = \langle \g, \o\rangle$ ($\s \in G_k$). 
We have Artin's comparison isomorphism \cite[XI, 4 and XVI, 4]{artin} between the \'etale and the Betti homology
\begin{equation*}
H_1^\et(X, \Z/n\Z)  \simeq H_1(X,\Z/n\Z)
\end{equation*}
compatible with the left $\D$-actions. 
Therefore, the left $\D$-module $H_1(X,\Z/n\Z)$ is endowed with a left $G_k$-action 
and the two actions commute if the action of $\D$ on $X$ is defined over $k$. 

More generally, for any curve $X$ embeddable in a smooth projective variety $P$ of dimension $d$, its \'etale homology is defined by 
\[H_1^{\et}(X,\Z/n\Z)=H_{X,\et}^{2d-1}(P,\Z/n\Z(d)).\] 
For a subvariety $Y \subset X$, the relative \'etale homology is defined by 
\[H_1^\et (X,Y;\Z/n\Z)=H^{2d-1} \mathrm{Cone}[R\Gamma_Y(P,\Z/n\Z(d)) \to R\Gamma_X(P,\Z/n\Z(d))]. \]
These groups admit $G_k$-actions when $X \inj P$ and $Y$ are defined over $k$. 

\subsection{Definition}\label{s3.3}

We apply the generalities in the preceding section to the hypergeometric curves $X_{N,\l}$. 
If $\l \in L(k)=k-\{0,1\}$, then $X_{N,\l}$ is defined over $k$ and the $\L$-module $H_1(X_{\infty,\l}, \wh\Z)$ admits a left $G_k$-action. 
Let 
\[\cyc\colon G_\Q \to \wh\Z^\times\] 
denote the cyclotomic character, i.e. 
$\s\z_N = \z_N^{\cyc(\s)}$ for any $\s\in G_\Q$ and $N$. 
The group $G_\Q$ acts continuously on $\wh\Z(1)$, hence on the ring $\L$, 
and the $G_k$-action on $H_1(X_{\infty,\l}, \wh\Z)$ is semilinear, 
i.e. $\s(c\g)=(\s c) (\s\g)$ for any $\s\in G_k$, $c\in\L$ and $\g \in H_1(X_{\infty,\l}, \wh\Z)$. 

For a free $\L$-module $V$ of rank $n$ equipped with a semilinear $G_k$-action and a $\L$-basis 
$v=\{v_1, \dots, v_n\}$ of $V$, define a map 
\[M^v\colon G_k \to \GL_n(\L)\] 
by 
$\s \bmat{v_1 & \cdots & v_n} = \bmat{v_1 & \cdots & v_n}M^v(\s)$. 

\begin{ppn}\label{p4.1}\ 
\begin{enumerate}
\item
The map $M^{v}$ is a $1$-cocycle, i.e. for any $\s, \t\in G_k$, 
\[M^{v}(\s\t)= M^{v}(\s) \cdot \s M^v(\t).\]
\item The cohomology class of $M^{v}$ in $H^1(G_k,\GL_n(\L))$ 
is independent of the choice of $v$. 
\item The restriction of the map $M^{v}$ to $G_{k\Q^\ab}$ is a group homomorphism. 
\item The restriction of the maps $\det M^v$ and $\tr M^{v}$ to $G_{k\Q^\ab}$ are independent of the choice of $v$. 
\end{enumerate} 
\end{ppn}

\begin{proof}
The proof of (i) is straight-forward. 
If another basis $u=\{u_1,\dots, u_n\}$ is given by $\bmat{u_1& \cdots & u_n}=\bmat{v_1 & \cdots & v_n}P$ with $P\in \GL_n(\L)$, 
then
\begin{equation}\label{e4.2}
M^u(\s)=P^{-1}M^v(\s)\s P.
\end{equation}
This is nothing but the coboundary condition of the non-abelian cohomology (cf. \cite[Chapitre VII, Annexe]{serre}), hence (ii). 
Since $G_{\Q^\ab}$ acts trivially on $\L$, (iii) and (iv) follow from (i) and \eqref{e4.2}, respectively. 
\end{proof}

We apply this to the $\L$-module $H_1(X_{\infty,\l},\wh\Z)$ of rank two (Corollary \ref{c4.2} (i)). 

\begin{dfn}\label{d4.1}
Choose a $\L$-basis $v$ of $H_1(X_{\infty,\l},\wh \Z)$, e.g. $v=\{\a_\l,\b_\l\}$ or $\{\a_\l,\g_\l\}$, and let 
\[M_\l^{v}\colon G_k \to \GL_2(\L)\]
be the associated cocycle. Define the {\it adelic hypergeometric function of argument $\l$} with respect to the basis $v$ by the map 
\[F_\l^{v}\colon G_k \to \L; \quad F_\l^{v}(\s)=\tr M_\l^{v}(\s).\]
\end{dfn}

We evaluate this function at characters of $\D_\infty$. 
We have an isomorphism
\[(\Q/\Z)^{\oplus 2} \simeq  \Hom_\mathrm{cont}(\D_\infty, (\Z^\ab)^\times); 
\quad 
(s,t)\mapsto \chi^{s,t},\]
where $\chi^{\frac{a}{N},\frac{b}{N}}$ is the composition of the canonical map $\D_\infty \to \D_N$ and $\chi_N^{a,b}$ (see Section \ref{s2.1}). 
For any $\Phi \in \L$, define a function as 
\[\Phi \colon (\Q/\Z)^{\oplus 2}\to  \wh\Z \ot \Z^\ab; \quad \Phi(s,t) = \int_{\D_\infty} \chi^{s,t} \, d\m_\Phi.\]
When $(s,t)=(a/N,b/N)$, the value $\Phi(s,t)$ is nothing but the image of $\Phi$ under the composition 
\[\L \to \wh\Z[\D_N] \xrightarrow{\chi_N^{a,b}} \wh\Z \ot \Z[\m_N].\]

\begin{ppn}\label{p4.7}
If $(s, t) \in (N^{-1}\Z/\Z)^{\oplus 2}$ and $\s\in G_{k(\mu_N)}$, then
$F_\l^{v}(\s)(s,t)$ is independent of the choice of a basis $v$. 
\end{ppn}

\begin{proof}The values of $F_\l^{v}(\s)$ on $(N^{-1}\Z/\Z)^{\oplus 2}$ depend only on the image of $F_\l^{v}(\s)$ under $\L \to \wh\Z[\D_N]$. Since $G_{k(\mu_N)}$ acts trivially on $\D_N$, 
the matrices $P$ and $\s P$ in \eqref{e4.2} have the same image in $\GL_2(\wh\Z[\D_N])$ 
and the image of 
$F_\l^{v}(\s)$ in $\wh\Z[\D_N]$ is independent of the basis, hence the proposition. 
\end{proof}

\begin{dfn}\label{d3.9}
For a subfield $k\subset \ol\Q$ and a positive integer $N$, put 
\[\sL_{k,N} = L(k) \times G_{k(\mu_N)} \times (N^{-1}\Z/\Z)^{\oplus 2},\]
and put 
\[\sL = \bigcup_{k,N} \sL_{k,N} \ \subset L(\ol\Q) \times G_\Q \times (\Q/\Z)^{\oplus 2}.\]
Define the {\it adelic hypergeometric function} by 
\[F \colon \sL \to \wh\Z \ot \Z^\ab; \quad (\l,\s, (s,t)) \mapsto F_\l(\s)(s,t),\] 
where we omit writing $v$ because of Proposition \ref{p4.7}. 
This is the adelic analogue of the complex hypergeometric functions $F(a,b;1;\l)$. 
\end{dfn}

\begin{ppn}\label{p3.9}\ 
\begin{enumerate}
\item We have the symmetry $F_\l(\s)(s,t)=F_\l(\s)(t,s)$.  
\item For any $\t\in G_\Q$, we have 
\[\t(F_\l(\s)(s,t))=(\t F_\l(\s))(s,t)=F_\l(\s)(\cyc(\t)s,\cyc(\t)t).\]
\end{enumerate}
\end{ppn}
\begin{proof}
(i) If $\iota$ is the involution of $X_{N,\l}$ as in Corollary \ref{c2.9}, 
we have
$F_\l^v(\s)(t,s)=F_\l^{\io_*v}(\s)(s,t)=F_\l^v(s,t)$ 
by the independence of a basis. 
(In fact, $F_\l^v\colon G_k \to \L$ already has this symmetry by Corollary \ref{c2.9}.)

(ii) This is immediate from the definition. 
\end{proof}

\begin{epl}
Suppose that $s+t=0$. If $\tau\in G_\Q$ denotes the complex conjugation, we have by Proposition \ref{p3.9} 
\[\tau(F_\l(\s)(s,t))=F_\l(\s)(-s,-t)=F_\l(\s)(t,s)=F_\l(\s)(s,t).\] 
Hence $F_\l(\s)(s,t)\in \wh\Z \ot (\Z^\ab \cap \R)$. 
In particular if  $s\in N^{-1}\Z/\Z$ with $N\in\{2,3,4,6\}$, then $F_\l(\s)(s,-s)\in \wh\Z$. 
\end{epl}

\subsection{Specializations}\label{s3.4}
We describe the Galois action on the homology of the fibers at $\l\in\{0,1\}$ (see Section \ref{s2.7}). 

First, we introduce ``cyclotomic units". 
For $x\in\{\x, \y, \x\y\}$, let 
\[C_x\colon G_\Q \to \L^\times\] 
be the cocycle associated with  the free cyclic $\L$-module $\L(1-x) \subset \L$ and the basis $1-x$, 
i.e. 
\[1-x^{\cyc(\s)}=C_x(\s) (1-x)\quad (\s\in G_\Q).\]
This definition makes sense by the following. 

\begin{lem}
For $x\in\{\x, \y, \x\y\}$, $\L (1-x)$ is a $G_\Q$-stable free cyclic $\L$-module. 
\end{lem}

\begin{proof}
By symmetry, we can assume $x=\x$. 
Suppose that $f (1-\x) =0$ ($f \in \L$) and for each $N$, let $f_N \in \wh\Z[\D_N]$ be the image of $f$. 
Since the annihilator of $1-\x_N$ in $\wh\Z[\D_N]$ is the ideal generated by $\sum_{i\in\Z/N\Z} \x_N^i$, 
it follows that $f_N \in  \wh\Z[\D_N] I_N$. 
By Lemma \ref{l3.1}, we obtain $f=0$, i.e. the freeness. For $\s \in G_\Q$, let $c=\cyc(\s) \in \wh\Z^\times$ and for each $N$, $c_N$ be a positive integer such that $c\equiv c_N\pmod{N}$.   Then  
\[\s(1-\x_N)=1-\x_N^{c_N}=\left(\sum_{i=0}^{c_N-1} \x_N^i\right) (1-\x_N) .\]
Since $f_N:=\sum_{i=0}^{c_N-1} \x_N^i$ is well-defined in $\Z[\D_N]/I_N$ independent of the choice of $c_N$ 
and $f:=(f_N)$ is an element of $\varprojlim \Z[\D_N]/I_N$, 
we  have $\s(1-\x)=f (1-\x)$ in $\L$ by Lemma \ref{l3.1}, hence the $G_\Q$-stability is proved. 
\end{proof}

Secondly, we recall Anderson's adelic beta function \cite{anderson} (cf. \cite[Section 7]{otsubo1}). 
Let $Y_N=C_N\cap\{z_0\ne 0\}$ be the affine Fermat curve (see Section \ref{s2.7}) and put $W_N=Y_N\cap\{x_0y_0=0\}$. 
Then the relative homology 
\[H_1(Y_\infty,W_\infty;\wh\Z):=\varprojlim_N H_1(Y_N,W_N;\wh\Z)\] 
is a free cyclic $\L$-module generated by $\d_\infty=(\d_N)_N$, where $\d_N$ is the specialization of the path $\d_{N,\e}$ (see Section \ref{s2.3}) to $\l=1$.  
Then the cocycle 
\[B\colon G_\Q \to \L^\times\]
associated with the basis $\d_\infty$, which can be viewed as a function $G_\Q \times (\Q/\Z)^{\oplus 2} \to \wh\Z \ot \Z^\ab$, is the adelic beta function. Its $l$-adic component for a prime number $l$, i.e. the image under $\L \to \Z_l[[\D_{l^\infty}]]$ (see Remark \ref{r3.1}), agrees Ihara's $l$-adic beta function \cite{ihara}.

\begin{thm}\label{t3.12}  \
\begin{enumerate}
\item
If $B_0\colon G_\Q \to \L^\times$ denotes the cocycle associated with the basis $\sp_0(\b_{\l})$ of $H_1(X_{\infty,0},\wh\Z)$, then 
\[B_0=C_\x C_\y.\]
\item
If $A_1 \text{ (resp. $B_1$)} \colon G_\Q \to \L^\times$ denotes the cocycle associated with the basis $\sp_1(\a_{\l})$ 
(resp. $\sp_1(\b_\l)$) of $H_1(X_{\infty,1},\wh\Z)$ (resp. $H_1(C_{\infty},\wh\Z)$), then  
\[C_{\x\y}A_1=B_1=C_\x C_\y B.\]
\end{enumerate}
\end{thm}

\begin{proof}
By Theorem \ref{t2.5} and its proof, we have an isomorphism of $\L$-modules 
\[H_1(X_{\infty,0},\wh \Z) \simeq \L(1-\x)(1-\y); \quad \sp_0(\b_\l) \mapsto (1-\x)(1-\y),\]
hence (i) follows.
By Theorem \ref{t2.4} and the definition of $\b_\l$, we have 
\[-(1-\x\y)\sp_1(\a_\l) = \sp_1(\b_\l)= (1-\x)(1-\y)\d\] 
in $H_1(U_\infty,W_\infty;\wh\Z)$, hence (ii) follows.  
\end{proof}

The statement (ii) can be seen as an adelic analogue of the Euler-Gauss summation formula \eqref{e0.1} and its finite analogue \eqref{e0.2}.

\section{Interpolation}
\subsection{Hypergeometric functions over finite fields}\label{s4.1} 

Let us recall from \cite{otsubo2} the definition and some properties of hypergeometric functions over finite fields. 
See \cite[Remark 2.12]{otsubo2} for definitions given by other authors and their relations with ours. 
Let $\k$ be a finite field of characteristic $p$ with $q$ elements. 
Fix a non-trivial additive character $\psi \in \Hom(\k,\C^\times)$, e.g. 
$\psi(x)=\be(\mathrm{Tr}_{\k/\F_p}(x)/p)$ where $\mathrm{Tr}_{\k/\F_p}\colon\k\to\F_p$ is the trace map. 
Let 
\[\ck :=\Hom(\k^\times,\C^\times)\]
be the group of multiplicative characters of $\k$, and $\e\in\ck$ be the unit character. 
We write $\ol\vp=\vp^{-1}$ and set $\vp(0)=0$ for any $\vp\in\ck$.

For $\vp \in \ck$, define the {\it Gauss sum} and its variant by 
\[g(\vp)=-\sum_{x\in\k} \psi(x)\vp(x), \quad g^\0(\vp)
=\begin{cases} g(\vp) & (\vp\ne\e), \\ q & (\vp=\e). \end{cases}\]
Note that $g(\e)=1$. 
Then we have for any $\vp\in\ck$
\begin{equation}\label{reflection-fin}
g(\vp)g^\0(\ol\vp)=q \vp(-1)
\end{equation}
 (cf. \cite[Proposition 2.2 (ii)]{otsubo2}). 
For $\vp_1$, $\vp_2\in\ck$, the {\it Jacobi sum} is defined by 
\[j(\vp_1,\vp_2)=-\sum_{x,y\in\k, x+y=1} \vp_1(x)\vp_2(y).\]
Then we have, 
\begin{equation}\label{j1}
j(\vp_1,\vp_2)= \frac{g(\vp_1)g(\vp_2)}{g^\0(\vp_1\vp_2)}
\end{equation}
unless $\vp_1=\vp_2=\e$  (cf. \cite[Proposition 2.2 (iv)]{otsubo2}).

For $\a, \n \in \ck$, define the Pochhammer symbol and its variant by 
\[(\a)_\n=\frac{g(\a\n)}{g(\a)}, \quad (\a)^\0_\n=\frac{g^\0(\a\n)}{g^\0(\a)}.\]
For $\a, \b, \g  \in \ck$, the Gaussian hypergeometric function over $\k$ is defined by 
\[F(\a,\b;\g;\l)= \frac{1}{1-q} \sum_{\n \in \ck}  \frac{(\a)_\n(\b)_\n}{(\e)^\0_\n(\g)^\0_\n} \n(\l) \quad (\l\in \k).\]
It takes values in $\Q(\m_{q-1})$ and is independent of the choice of $\psi$ \cite[Definition 2.7]{otsubo2}. 
Note that $F(\a,\b;\g;0)=0$ by definition.   

For simplicity, we assume from now on that $\a,\b\not\in \{\e,\g\}$. 
Then we have a finite analog of the integral representation \eqref{hgf-i} (see \cite[Corollary 3.5 (i)]{otsubo2}) 
\begin{equation}\label{e3.3}
-j(\b, \ol\b\g) F(\a,\b;\g;\l)= \sum_{t\in\k}  \ol\a(1-\l t) \b(t) \ol\b\g(1-t).
\end{equation}
From this follows the finite analogue of the Euler-Gauss summation formula \eqref{euler-gauss} (see  \cite[Theorem 4.3]{otsubo2}) 
\begin{equation}\label{euler-gauss-fin}
F(\a,\b;\g;1)=
\frac{g^\0(\g)g(\ol{\a\b}\g)}{g(\ol\a\g)g(\ol\b\g)}. 
\end{equation}
In particular if $\g=\e$, letting $u=\frac{1}{1-\l t}$ and $t=\frac{v}{v-1}$ in \eqref{e3.3} and using $j(\b,\ol\b)=\b(-1)$ by \eqref{reflection-fin} and \eqref{j1}, we have 
\begin{equation}\label{hgf-fin-i2}
F(\a,\b;\e;\l)= -\sum_{u,v \in \k, (1-u)(1-v)=\l uv} \a(u)\b(v). 
\end{equation}
In particular, $F(\a,\b;\e;\l) \in \Z[\m_{q-1}]$. Letting $\l=1$, we obtain \eqref{e0.2}. 

\begin{rmk}
The parameter set $\ck \simeq (q-1)^{-1}\Z/\Z$ is an analogue of the rational numbers with denominator dividing $q-1$. 
An integral shift of complex parameters, which corresponds to a first order differential operation (see for example \eqref{hgf-d}), makes no change over finite fields. 
Finite analogues of linearly independent solutions of \eqref{hgf-de} are proportional; for example
\[F(\a,\b;\e;\l)= j(\a,\b) F(\a,\b;\a\b;1-\l)\]
for any $\l \ne 0, 1$ and $\a, \b \ne \e$ (see \cite[Theorem 3.15]{otsubo2}). 
\end{rmk}

\subsection{Number of rational points on $X_{N,\l}$}\label{s4.2}

Let $\k$ be as above, $\ol\k$ be an algebraic closure of $\k$, and $\Fr \in \mathrm{Gal}(\ol\k/\k)$ be the $q$th power Frobenius element. 
For a smooth variety $X$ over $\k$ equipped with an action (over $\k$) of a finite abelian group $\D$, 
the number of $\k$-rational points on $X$ decomposes as 
\begin{equation*}
\# X(\k) = \sum_{\chi \in \wh \D} N(X,\chi),
\end{equation*}
where
\[N(X,\chi)=\frac{1}{\# \D} \sum_{g \in \D} \chi(g) \# \{P \in X(\ol\k) \mid \Fr(P)=gP\}.\]
By the Grothendieck-Lefschetz trace formula, we have 
\begin{equation}\label{e3.1} 
N(X,\chi)=\sum_{i=0}^{2 \dim X} (-1)^i \tr\bigl(\Fr^{-1}; H^i(X)(\chi)\bigr),
\end{equation}
where 
\[H^i(X):=H_{\et,c}^i(X \ot_\k\ol\k,\ol\Q_l)\]
is the $l$-adic \'etale cohomology with compact support for a prime $l\ne p$, and its $\chi$-part $H^i(X)(\chi)$ is defined similarly as in Section \ref{s2.2}.
Here we fixed an embedding $\ol\Q \hookrightarrow \ol\Q_l$ and viewed $\chi$ as a $\ol\Q_l$-valued character.  

Now assume that $q\equiv 1 \pmod{N}$, so that $\k$ contains all the $N$th roots of unity. 
Let $\l\in \k-\{0,1\}$  and  $X_{N,\l}$ be the hypergeometric curve over $\k$ defined as in Section \ref{s2.1}. 
Fix an isomorphism 
\[\io \colon \mu_N(\C) \os\simeq\lra \mu_N(\k)\]
 and let $\D_N$ act on $X_{N,\l}$ over $\k$ through $\io$. 
Then the number of rational points on $X_{N,\l}$ is written in terms of hypergeometric functions over $\k$ as follows. 
Define $\vp_N \in \wh{\k^\times}$ of exact order $N$ by $\vp_N(x)= \io^{-1} (x^\frac{q-1}{N})$. 

\begin{thm}\label{t3.2} 
Suppose that $a, b \ne 0$. Then we have 
\[\tr (\Fr^{-1}; H^1(X_{N,\l})(\chi_N^{a,b})) =-N(X_{N,\l};\chi_N^{a,b})=F(\vp_N^a,\vp_N^b;\e;\l)\]
in $\Q(\m_N)$.
\end{thm}

\begin{proof}
We have 
\begin{align*}
H^i(X_{N,\l}) &=H^i(X_{N,\l})(\chi_N^{0,0}) \quad (i=0,2), \\
H^1(X_{N,\l}) & =\bigoplus_{a, b \ne 0} H^1(X_{N,\l})(\chi_N^{a,b}). 
\end{align*}
The first follows since $\D_N$ acts trivially on $H^i(X_{N,\l})$ for $i=0, 2$. 
The second is proved for example by reducing to characteristic $0$ and use Proposition \ref{p2.2}. 
Hence the first equality of the theorem follows by \eqref{e3.1}. 
To prove the second equality, let $Y_{N,\l}=\A^2\cap  X_{N,\l}$ be the affine part. 
For $(x,y) \in Y_{N,\l}(\ol\k)$, consider the condition 
$\Fr(x,y)=(\z_N^i,\z_N^j)(x,y)$.  
If we put $u=x^N$, $v=y^N$, then $(1-u)(1-v)=\l uv$, and the condition is  
equivalent to that $u, v \in \k^\times$, $\vp_N(u)=\z_N^i$ and $\vp_N(v)=\z_N^j$. 
Therefore, 
\[N(Y_{N,\l};\chi_N^{a,b})= \sum_{u, v \in \k,  (1-u)(1-v)=\l uv} \vp_N^a(u)\vp_N^b(v)=-F(\vp_N^a,\vp_N^b;\e;\l)\]
by \eqref{hgf-fin-i2}. 
Let $P=(x,\infty) \in X_{N,\l}(\ol\k)$ be a point at infinity, i.e. $x^N=(1-\l)^{-1}$. 
Then the condition $\Fr(P)=(\z_N^i,\z_N^j)P$ is equivalent to $\vp_N((1-\l)^{-1})=\z_N^i$, which is independent of $x$ and $j$. Since $\sum_j \z_N^{bj}=0$,  these points have no contribution to $N(X_{N,\l};\chi_N^{a,b})$, 
and similarly for the points $(\infty, y)$ ($y^N=(1-\l)^{-1}$). 
Hence $N(X_{N,\l};\chi_N^{a,b})=N(Y_{N,\l};\chi_N^{a,b})$ and the theorem is proved. 
\end{proof}

\begin{rmk}
The proof as above is a generalization of Weil's point counting on the Fermat curves \cite{weil}.  
By counting points on the curves $C_{N,\l}^{a,b}$ (see Remark \ref{r1}) and their higher-dimensional analogues, Koblitz \cite{koblitz} arrived at the first definition of hypergeometric functions of type ${}_{d+1}F_d$ over finite fields. 
\end{rmk}

\subsection{Interpolation}
 
Let $k$ be a number field containing $\mu_N$. 
For a prime $v \nmid N$ of $k$, let $\k_v$ denote the residue field of $v$ and put $q_v=\# \k_v$.  Let $\vp_{N,v} \in \wh{\k_v^\times}$ be the $N$th power residue character at $v$, i.e. 
\[\vp_{N,v}(x)\equiv x^\frac{q_v-1}{N} \pmod v.\]
For a prime number $l$, let 
\[\pi_l\colon \wh\Z \ot \Z^\ab \to \Z_l \ot \Z^\ab\]  
denote the projection. 

Recall that the adelic beta function interpolates all the Jacobi sums \cite[Section 10]{anderson}, \cite[II, Theorem 7]{ihara} (see also \cite[Section 8]{otsubo1}). 
If $\s_v \in G_k$ is a lift of Frobenius at $v$, then 
\begin{equation}\label{e4.14}
\pi_l\left(B(\s_v)\left(\frac{a}{N},\frac{b}{N}\right)\right)= 1 \ot j(\vp_{N,v}^a,\vp_{N,v}^b)
\end{equation}
for any $a, b \in \Z/N\Z-\{0\}$ with $a+b \ne 0$, and a prime $l\nmid q_v$. 

The adelic hypergeometric function interpolates the hypergeometric functions over finite fields. 

\begin{thm}\label{t4.8}
Suppose that $\l \in L(k)$ is integral at $v$ and $v\nmid \l(1-\l)N$. 
If $\s_v \in G_k$ is a lift of Frobenius at $v$ and $a, b \in \Z/N\Z-\{0\}$, then 
\[\pi_l\left(F_\l(\s_v)\left(\frac{a}{N},\frac{b}{N}\right)\right)= 1\ot  F\left(\vp_{N,v}^a, \vp_{N,v}^b;\e; \l \text{ {\rm mod} } v\right)\]
for any prime number $l\nmid q_v$. 
(See Theorem \ref{t4.14} below for the cases when $a=0$ or $b=0$.) 
\end{thm}

\begin{proof}
By the assumption $v \nmid \l(1-\l)N$, the model of $X_{N,\l}$ over $k$ has good reduction at $v$. 
By the proper and smooth base change theorems \cite[Arcata, V, 3]{4demi}, 
the $l$-adic \'etale (co)homology of $X_{N,\l}$ agrees that of the reduction and the inertia subgroup $I_v \subset G_k$ of $v$ acts on it trivially.  
Under the Poincar\'e duality paring \eqref{e4.1} with $\ol\Q_l$-coefficients, 
$H_1(X_{N,\l})(\chi_N^{a,b})$ is dual to $H^1(X_{N,\l})(\chi_N^{a,b})$. 
Since the matrix $M_\l^v(\s_v)$ (see Definition \ref{d4.1}) representing $\s_v$ on the former is the matrix representing $\s_v^{-1}$ on the latter up to conjugates, 
the theorem follows by Theorem \ref{t3.2}.
\end{proof}

\begin{rmk}\label{r4.5}
A motivic proof of Theorem \ref{t4.8} will be given in a forthcoming paper of the second author, 
where the Frobenius endomorphism of the motive of $X_{N,\l}$ over a finite field is expressed in terms of a ``hypergeometric algebraic correspondence''.  
\end{rmk}

\subsection{Transformation and summation formulas}\label{s4.4}

Recall the transformation formulas of Euler and Pfaff for complex hypergeometric functions (cf. \cite[2.1.4]{erdelyi}, \cite[Section 2.4]{otsubo4}):
\begin{align*}
& F(a,b;c;\l)= (1-\l)^{c-a-b} F(c-a,c-b;c;\l),
\\ & F(a,b;c;\l)=(1-\l)^{-a} F\left(a,c-b;c;\frac{\l}{\l-1}\right).
\end{align*}
Their finite field analogues are (cf. \cite[Theorem 3.14]{otsubo2}):  
\begin{align}
&F(\a,\b;\g;\l)=\ol{\a\b}\g(1-\l) F(\ol\a\g,\ol\b\g;\g;\l), \label{e4.3}
\\
&F(\a,\b;\g;\l)=\ol\a(1-\l)F\left(\a,\ol\b\g;\g;\frac{\l}{\l-1}\right), \label{e4.4}
\end{align}
where $\a, \b \not\in\{\e,\g\}$, $\l\in\k-\{1\}$. 

To state the adelic analogues, we introduce some notations. 
Let 
\[k^\times \to H^1(G_k,\m_N)\] 
be the Kummer map and write the image of $\l$ by $K_{\l,N}$. 
The restriction of $K_{\l,N}$ to $G_{k(\mu_N)}$ is a homomorphism. 
For any $s=a/N \in N^{-1}\Z/\Z$, define a Galois character as
\[K_\l^s\colon G_{k(\m_N)} \to \m_N; \quad K_\l^s(\s) = K_{\l,N}(\s)^a, \]
which is independent of the presentation of $s$. 
If $v \nmid N$ is a prime of $k(\mu_N)$ and $\l$ is a $v$-unit, then we have by definition
\begin{equation}\label{e4.5}
K_\l^\frac{a}{N} (\s_v)=\vp_{N,v}^a(\l \text{ mod } v)
\end{equation}
for any lift $\s_v$ of Frobenius at $v$. 

The adelic analogues of the Euler and Pfaff formulas are as follows. 
\begin{thm}\label{t4.9}
As functions on $\sL$, we have 
\begin{align*}
& F_\l(\s)(s,t)=K_{1-\l}^{-s-t}(\s) F_\l(\s)(-s,-t), 
\\
& F_\l(\s)(s,t)=K_{1-\l}^{-s}(\s) F_{\frac{\l}{\l-1}}(\s)(s,-t). 
\end{align*}
\end{thm}

\begin{proof}
It suffices to compare the images under $\pi_l$ for each prime number $l$. 
Let $\l\in L(k)$ and $(s,t)=(a/N,b/N)$. By the Chebotarev density theorem, it suffices to prove the formula for Frobenius lifts $\s_v \in G_{k(\mu_N)}$, where $v$ is a prime of $k(\mu_N)$ such that $\l$ is $v$-integral and $v \nmid \l(1-\l)Nl$. 
By Theorem \ref{t4.8} and \eqref{e4.5}, the formulas for $s, t \ne 0$ reduce to \eqref{e4.3} and \eqref{e4.4} by letting $\a=\vp_{N,v}^a$, $\b=\vp_{N,v}^b$ and $\g=\e$. 
The cases when $s=0$ or $t=0$ will be proved in Section \ref{s5.1}. 
\end{proof}

\begin{rmk}\label{r4.7}
We can also derive Theorem \ref{t4.9} from the motivic Euler and Pfaff transformation formulas (see Remark \ref{r4.5}). 
\end{rmk}

Over the complex numbers, some quadratic and resulting quartic transformation formulas are known by Gauss and Kummer 
(cf. \cite[2.1.5]{erdelyi}, \cite[Section 4]{otsubo4}). 
Their finite field analogues are known (see \cite[Theorem 5.1, Corollary 5.3, Theorem 5.4, Theorem 5.5 and Corollary 5.6]{otsubo2}). 
By a similar argument as in the proof of the previous theorem, we obtain the following adelic analogues.  
See Section \ref{s5.1} below for the case $2s=0$ in \eqref{e4.6}--\eqref{e4.9} and the case $s=0$ in \eqref{e4.10}. 

\begin{thm}\label{t4.10} 
We have the following equalities as far as the both sides are defined:
\begin{align}
& K_{1+\l}^{2s}(\s) F_\l(\s)(2s,2s)=F_{1-\left(\frac{1-\l}{1+\l}\right)^2}(\s)\left(s,s+\frac{1}{2}\right), \label{e4.6}
\\& K_{1+\l}^{2s}(\s) F_{\l^2}(\s)\left(s, s+\frac{1}{2}\right)
=F_{1-\frac{1-\l}{1+\l}}(\s)\left(2s,\frac{1}{2}\right),   \label{e4.7}
\\& F_\l(\s)(2s,-2s)=F_{1-(1-2\l)^2}(\s)\left(s,\frac{1}{2}-s\right),  \label{e4.8}
\\& 
K_{1+\l}^{2s}(\s)F_{-\l}(\s)(2s,2s)=F_{1-\left(\frac{1-\l}{1+\l}\right)^2}(\s)\left(s,\frac{1}{2}-s\right),  \label{e4.9}
\\&
K_{1+\l}^{2s}(\s) F_{\l^2}(\s)(s,s)=F_{1-\left(\frac{1-\l}{1+\l}\right)^2}(\s)\left(s,\frac{1}{2}\right),  \label{e4.10}
\\&
K_{1+3\l}^\frac{1}{2}(\s) F_{\l^2}(\s)\left(\frac{1}{4}, \frac{3}{4}\right)
=F_{1-\left(\frac{1-\l}{1+3\l}\right)^2}(\s)\left(\frac{1}{4},\frac{3}{4}\right),   \label{e4.11}
\\&
F_{\l^4}(\s)\left(\frac{1}{2}, \frac{1}{2}\right)
=F_{1-\left(\frac{1-\l}{1+\l}\right)^4}(\s)\left(\frac{1}{2},\frac{1}{2}\right),   \label{e4.12}
\\&
F_{-\l^2}(\s)\left(\frac{1}{2}, \frac{1}{2}\right)=
F_{1-\left(\frac{1-\l}{1+\l}\right)^4}(\s)\left(\frac{1}{4},\frac{1}{2}\right).   \label{e4.13}
\end{align}
\qed
\end{thm}

For example, \eqref{e4.10} is an analogue of Gauss's formula 
\[(1+\l)^{2a }F\left(a,a;1;\l^2\right) = F\left(a,\frac{1}{2};1;1-\left(\frac{1-\l}{1+\l}\right)^2 \right).\]
There is a cubic analogue  
\[(1+2\l) F\left(\frac{1}{3},\frac{2}{3};1;\l^3\right)=F\left(\frac{1}{3},\frac{2}{3};1;1-\left(\frac{1-\l}{1+2\l}\right)^3\right)\]
proved by Borwein-Borwein \cite[p. 694]{borweins}. (See \cite{otsubo4} for simple proofs of these formulas.)  
A finite analogue of the latter is proved by Fuselier et. al. \cite[(9.1)]{fuselieretal}, from which we can deduce as before 
the following adelic analogue. 

\begin{thm}
For any $\l \in L(k)$ with $\l^3 \ne 1$, $1+2\l\ne 0$ and any $\s \in G_{k(\m_3)}$, 
\[F_{\l^3}(\s)\left(\frac{1}{3},\frac{2}{3}\right) = F_{1-\left(\frac{1-\l}{1+2\l}\right)^3}(\s) \left(\frac{1}{3},\frac{2}{3}\right) .\]
\qed
\end{thm}

Recall Kummer's summation formula for the complex hypergeometric function (cf. \cite[2.8 (47)]{erdelyi})
\[F(2a,b;2a-b+1;-1)=\frac{\G(2a-b+1)\G(a+1)}{\G(2a+1)\G(a-b+1)}.\]
In particular, 
\[F(2a,2a;1;-1) = \frac{\sin \pi a}{2\pi} B(a,a).\]
Its finite analogue is as follows. If $\k$ is a finite field of odd characteristic, then  
\[F(\a^2,\b;\a^2\ol\b;-1)=\sum_{\a'\in\ck,\a'^2=\a^2} \frac{g^\0(\a^2\ol\b) g(\a')}{g(\a^2)g^\0(\a'\ol\b)}\] 
for any $\a, \b \in \ck$ \cite[theorem 4.6 (ii)]{otsubo2}. 
In particular, 
\[F(\a^2,\a^2;\e;-1) = \sum_{\a'\in\ck, \a'^2=\a^2} \a'(-1) j(\a',\a') \]
for any $\a$. 
From this and Theorem \ref{t4.8}, one derives the following adelic analogue similarly as in the proof of Theorem \ref{t4.9} (see Section \ref{s5.1} below for the case $2s=0$). 

\begin{thm}\label{t4.12} 
If $s \in N^{-1}\Z/\Z$ and $\s \in G_{\Q(\m_N)}$, then 
\[F_{-1}(\s)(2s,2s) = K_{-1}(\s)^s B(\s)(s,s) + K_{-1}(\s)^{s+\frac{1}{2}} B(\s)\left(s+\frac{1}{2},s+\frac{1}{2}\right).\]
\qed
\end{thm}

\section{Complements}
\subsection{Open curves}\label{s5.1}

By our construction, the values of $F_\l(\s)(s,t)$ when $s=0$ or $t=0$ cannot be recovered from the homology of $X_{N,\l}$. To remedy this, we introduce open curves and relative homology similarly as in the case of Fermat curves (see Section \ref{s3.4}). 

For $\l \in L(\C)$, let $Y_{N,\l} \subset X_{N,\l}$ be the affine part, put $Z_{N,\l}=X_{N,\l}-Y_{N,\l}$ and let $W_{N,\l} \subset Y_{N,\l}$ be the closed subvariety defined by $xy=0$. 
We have exact sequences of $\Z[\D_N]$-modules 
\begin{equation}\label{e5.1}
0\to H_2(X_{N,\l},\Z) \to H_0(Z_{N,\l},\Z) \to H_1(Y_{N,\l},\Z) \to H_1(X_{N,\l},\Z) \to 0, 
\end{equation}
\begin{equation}\label{e5.2}
0\to H_1(Y_{N,\l},\Z) \to H_1(Y_{N,\l},W_{N,\l};\Z) \os{\pd}\to H_0(W_{N,\l},\Z) \to H_0(Y_{N,\l},\Z) \to 0. 
\end{equation}
Since both $Z_{N,\l}$ and $W_{N,\l}$ consist of $2N$ points, one easily computes that 
\[\rank H_1(Y_{N,\l},\Z) =2N^2-2N+1, \quad \rank H_1(Y_{N,\l},W_{N,\l};\Z) =2N^2.\] 
Tensoring \eqref{e5.1}, \eqref{e5.2} with $\C$, and decomposing them by the $\D_N$-action, we obtain
\begin{align*}
& H_2(X_{N,\l},\C)=H_2(X_{N,\l},\C)(\chi_N^{0,0}), \\
& H_0(Z_{N,\l},\C)=\bigoplus_{b\in\Z/N\Z} H_0(Z_{N,\l}^1,\C)(\chi_N^{0,b}) \oplus \bigoplus_{a\in\Z/N\Z} H_0(Z_{N,\l}^2,\C)(\chi_N^{a,0}), \\
& H_0(W_{N,\l},\C)=\bigoplus_{b\in\Z/N\Z} H_0(W_{N,\l}^1,\C)(\chi_N^{0,b}) \oplus \bigoplus_{a\in\Z/N\Z} H_0(W_{N,\l}^2,\C)(\chi_N^{a,0}), \\
& H_0(Y_{N,\l},\C)=H_0(Y_{N,\l},\C)(\chi_N^{0,0}), 
\end{align*}
where $Z_{N,\l}^1=\{x=\infty\}$, $Z_{N,\l}^2=\{y=\infty\}$, $W_{N,\l}^1=\{x=0\}$, $W_{N,\l}^2=\{y=0\}$, 
and all the eigenspaces in the right members are one-dimensional.

Let $\a_{N,\l}, \b_{N,\l} \in H_1(Y_{N,\l},\Z)$ and $\d_{N,\l} \in H_1(Y_{N,\l},W_{N,\l};\Z)$ 
denote the classes of the loops and the path, respectively, defined in Section \ref{s2.3}. 
Recall that $\b_{N,\l}=((1-\x_N)(1-\y_N))_*\d_{N,\l}$. 

\begin{thm}\label{t5.1}\ 
\begin{enumerate}
\item $H_1(Y_{N,\l},\Z)$ is generated by $\a_{N,\l}$ and $\b_{N,\l}$ as a $\Z[\D_N]$-module. 
\item $H_1(Y_{N,\l},W_{N,\l};\Z)$ is a free $\Z[\D_N]$-module generated by $\a_{N,\l}$ and $\d_{N,\l}$. 
\end{enumerate}
\end{thm}

\begin{proof}
Consider the $1$-form $\o_{N,\l}^{a,b}$  as defined in Section \ref{s2.2} with $a, b \in \{1,\dots, N\}$. 
It is holomorphic on $X_{N,\l}$ except for having logarithmic poles at $Z_{N,\l}^1$ (resp. $Z_{N,\l}^2$) if $a=N$ (resp. $b=N$) (see \eqref{e2.4}). 
Hence it defines a class of $H^1(Y_{N,\l},\C)(\chi_N^{a,b})$. 
Let $\g_1, \g_2 \in H_1(Y_{N,\l},\Z)$ denote the classes of positively-oriented small loops around the points 
$(\infty,(1-\l)^{-\frac{1}{N}}) \in Z_{N,\l}^1$ and $((1-\l)^{-\frac{1}{N}},\infty) \in Z_{N,\l}^2$, respectively. Then we have by \eqref{e2.4}
\begin{align*}
\frac{1}{2\pi i}\int_{\g_1} \o_{N,\l}^{N,b} &= N(1-\l)^{-\frac{b}{N}}, &\frac{1}{2\pi i} \int_{\g_1} \o_{N,\l}^{a,N} &= 0 \ (a \ne N), 
\\ 
\frac{1}{2\pi i} \int_{\g_2} \o_{N,\l}^{N,b} &= 0 \ (b \ne N),& \frac{1}{2\pi i}\int_{\g_2} \o_{N,\l}^{a,N} &= -N(1-\l)^{-\frac{a}{N}}. 
\end{align*}
It follows that the $2N-1$ classes $\o^{a,b}_{N,\l}$ ($a=N$ or $b=N$) are non-trivial and give a basis of $\Coker (H^1(X_{N,\l},\C) \to H^1(Y_{N,\l},\C))$.  
On the other hand, we have 
\[\frac{1}{2\pi i}\int_{\a_{N,\l}} \o^{N,b} = (1-\l)^{-\frac{b}{N}}, \quad \frac{1}{2\pi i}\int_{\a_{N,\l}} \o^{a,N} = (1-\l)^{-\frac{a}{N}}\]
by the computation of Theorem \ref{t2.2}. It follows that 
\[\left(\sum_{i=0}^{N-1} \x_N^i\right)_* \a_{N,\l}=\g_1, \quad \left(\sum_{i=0}^{N-1} \y_N^i\right)_* \a_{N,\l}= -\g_2\]
since both sides agree as functionals on $H^1(Y_{N,\l},\C)$. 
Note that $\sum_{i=0}^{N-1} \x_N^i$ and $\sum_{i=0}^{N-1} \y_N^i$ annihilates $H_1(X_{N,\l},\Z)$, 
and $\g_1$ and $\g_2$ are trivial in $H_1(X_{N,\l},\Z)$. 
Since $\g_1$ and $\g_2$ generate $\Im(H_0(Z_{N,\l},\Z) \to H_1(Y_{N,\l},\Z))$ as a $\Z[\D_N]$-module, 
we obtain (i) from Theorem \ref{t2.1}. 

Since $\pd(\d_{N,\l})=[(1,0)]-[(0,1)]$ generates 
$\Ker(H_0(W_{N,\l},\Z) \to H_0(Y_{N,\l},\Z))$ as a $\Z[\D_N]$-module, 
(ii) implies that 
$H_1(Y_{N,\l},W_{N,\l};\Z)$ is generated over $\Z[\D_N]$ by $\a_{N,\l}$ and $\d_{N,\l}$, 
and  (i) follows by comparing the ranks. 
\end{proof}

\begin{cor}\label{c5.2}
Let the notations be as in Section \ref{s3.1}. 
\begin{enumerate}
\item $H_1(Y_{\infty,\l},\wh \Z)$ is a free $\L$-module of rank two generated by $\a_\l$ and $\b_\l$, and we have an isomorphism 
$H_1(Y_{\infty,\l},\wh \Z) \xrightarrow\simeq H_1(X_{\infty,\l},\wh \Z)$. 
\item $H_1(Y_{\infty,\l},W_{\infty,\l};\wh \Z)$ is a free $\L$-module of rank two generated by $\a_\l$ and $\d_\l$. 
\end{enumerate}
\end{cor}
\begin{proof}
The statement (ii) is immediate from Theorem \ref{t5.1} (ii), from which follows the freeness in (i), since $\b_\l=((1-\x)(1-\y))_*\d_\l$ and $(1-\x)(1-\y) \in \L$ is not a zero-divisor. Then the isomorphism in (i) follows by Corollary \ref{c4.2} (i) and Theorem \ref{t5.1} (i). 
\end{proof}

%

Now suppose that $k \subset \ol\Q$ and $\l \in L(k)$. 
Let 
\[M_\l\colon G_k \to \GL_2(\L) \quad (\text{resp. $\wt{M}_\l \colon G_k \to \GL_2(\L)$})\] 
be the cocycle defined by 
$H_1(Y_{\infty,\l},\wh \Z)$ (resp. $H_1(Y_{\infty,\l},W_{\infty,\l};\wh \Z)$) 
with respect to the $\L$-basis given in Corollary \ref{c5.2} 
(see Section \ref{s3.3}). 

\begin{ppn}\label{p5.3}
If $\s \in G_{k(\m_N)}$ and $s, t \in N^{-1}\Z/\Z$, then we have 
\[\tr M_\l(\s)(s,t)=\tr \wt{M}_\l(\s)(s,t). \]
\end{ppn}
\begin{proof}
Since $\b_\l=((1-\x)(1-\y))_*\d_\l$, we have
\[\bmat{1&0\\0& (1-\x)(1-\y)} M_\l(\s) = \wt{M}_\l(\s) \bmat{1&0\\0& \s((1-\x)(1-\y))}. \]
Then follows 
\[\tr \wt{M}_\l(\s)-\tr M_\l(\s) =(1-C_\x(\s) C_\y(\s)) \wt{m}_{2,2}(\s),\]
where $\wt{m}_{2,2}(\s)$ is the $(2,2)$-component of $\wt{M}_\l(\s)$ and 
$C_\x(\s)$, $C_\y(\s)\in \L^\times$ are as defined in Section \ref{s3.4}. 
Since $C_\x(\s)(s,t)=C_\y(\s)(s,t)=1$, the proposition follows. 
\end{proof}

By Artin's comparison theorem, the exact sequences  \eqref{e5.1}, \eqref{e5.2} tensored with $\Q_l$ admit $G_k$-actions (see Section \ref{s3.2}).  
The $\chi_N^{a,b}$-components of the relative homology when $a=0$ or $b=0$ are described as follows. 

\begin{ppn}\label{p4.13}Suppose that $\m_N \subset k$. 
Then we have exact sequences of $\ol\Q_l$-modules with $G_k$-action  
\begin{align*}
0 \to \ol\Q_l (1)(K_{1-\l}^{-\frac{b}{N}}) \to & H_1(Y_N,W_{N,\l};\ol\Q_l)(\chi_N^{b,0}) \to \ol\Q_l \to 0, 
\\0 \to \ol\Q_l (1)(K_{1-\l}^{-\frac{a}{N}}) \to & H_1(Y_N,W_{N,\l};\ol\Q_l)(\chi_N^{a,0}) \to \ol\Q_l \to 0,
\end{align*}
for any $a, b \in \Z/N\Z$. 
Here, for a character $K$ of $G_k$, $\ol\Q_l(1)(K)$ is a one-dimensional $\ol\Q_l$-vector space on which $G_k$ acts by the product character $\cyc \cdot K$. 
\end{ppn}

\begin{proof}
First, $H_2(X_{N,\l},\ol\Q_l)=\ol\Q_l(1)$ and $H_0(Y_{N,\l},\ol\Q_l)=\ol\Q_l$. 
Secondly, $W_{N,\l}$ splits into $2N$ points over $k$ and each eigencomponent of 
$H_0(W_{N,\l})$ is isomorphic to $\ol\Q_l$. 
On the other hand, $Z_{N,\l}$ splits into $2N$ points only over $k((1-\l)^\frac{1}{N})$, and $G_k$ acts on $H_0(Z_{N,\l}^1,\ol\Q_l)(\chi_N^{0,b})$ (resp. $H_0(Z_{N,\l}^2,\ol\Q_l)(\chi_N^{a,0})$) by the character $K_{1-\l}^{-\frac{b}{N}}$ (resp. $K_{1-\l}^{-\frac{a}{N}}$). 
Note that in the $l$-adic version of the exact sequence \eqref{e5.1}, the second non-trivial term is replaced by the Tate twist $H_0(Z_{N,\l},\ol\Q_l)(1)$. 
Finally by Proposition \ref{p2.2}, we have $H_1(X_{N,\l},\ol\Q_l)(\chi_N^{a,b})=0$ if $a=0$ or $b=0$. 
Hence the proposition follows. 
\end{proof}

\begin{thm}\label{t4.14}
Let $\l$, $v$, $\s_v$ and $l$ be as in Theorem \ref{t4.8}. Then we have 
\begin{align*}
\pi_l\left(F_\l(\s_v)\left(0,\frac{b}{N}\right)\right) &=1 \ot \left(1+q_v\vp_{N,v}^{-b}(1-\l \text{ {\rm mod }} v) \right),\\
\pi_l\left(F_\l(\s_v)\left(\frac{a}{N},0\right)\right) &= 1 \ot \left(1+q_v \vp_{N,v}^{-a}(1-\l \text{ {\rm mod }} v) \right),
\end{align*}
for any $a, b \in \Z/N\Z$. 
Moreover, the formula of Theorem \ref{t4.8} remains true for all $a, b \in \Z/N\Z$. 
\end{thm}

\begin{proof}
The proof of the three equalities is parallel to the proof of Theorem \ref{t4.8}: it follows by Propositions \ref{p5.3}, \ref{p4.13} and \eqref{e4.5}. 
For the last statement, recall that over the complex numbers we have
$F(a,1;1;\l)=(1-\l)^{-a}$ by the binomial theorem. 
Its analogue over a finite field $\k$ of $q$ elements is that 
\[F(\a,\e;\e;\l)=1+q \ol\a(1-\l) \quad (\l\in\k-\{0,1\})\]
for any $\a \in \ck$ \cite[Example 3.9 (iii)]{otsubo2}.  
Applying this to $\a=\e$, $\vp_{N,v}^a$ and $\vp_{N,v}^b$, we obtain the result. 
\end{proof}

By Theorem \ref{t4.14} and the Chebotarev density theorem, we can prove the remaining exceptional cases of  Theorems \ref{t4.9} \ref{t4.10} and \ref{t4.12}. 
For Theorem \ref{t4.10} \eqref{e4.7}, note that 
$\vp_{2,v}(1-\l^2)=\vp_{2,v}\bigl(\frac{1+\l}{1-\l}\bigr)$. 
For Theorem \ref{t4.12}, one needs 
\[\pi_l\left(B(\s_v)(0,0)\right)=1 \ot 1, \quad 
\pi_l\left(B(\s_v)\left(\frac{1}{2},\frac{1}{2}\right)\right)=1 \ot \vp_{2,v}(-1) q_v,\]
which are not covered by \eqref{e4.14}. 
These follow similarly as Theorem \ref{t4.14} from the construction of $B$ (see Section \ref{s3.4}). 

\subsection{A conjecture}\label{s5.2}
 
To conclude, let us discuss briefly the higher-dimensional generalization of the adelic hypergeometric function. Recall the generalized hypergeometric function over the complex numbers 
\[\FFF{p}{q}{a_1,\dots, a_p}{b_1,\dots, b_q}{\l} = \sum_{n=0}^\infty \frac{\prod_{i=1}^p (a_i)_n}{(1)_n \prod_{j=1}^q (b_j)_n} \l^n. \]
Its finite analogue in full generality is defined in \cite{otsubo2}. 
Here we restrict ourselves to the case where $b_j=1$ for all $j$ and $p=q+1$. 
To construct the adelic analogue of $\FFF{d+1}{d}{a_0,\dots, a_d}{1,\dots, 1}{\l}$, we propose the following. 

\begin{conj}
For positive integers $d$, $N$ and $\l \in L(\C)$, consider the smooth hypersurface $U_{N,\l}^{(d)} \subset \A^{d+1}$ defined by 
\[\prod_{i=0}^d (1-x_i^N) = \l,\]
on which the group $\m_N^{d+1}$ acts naturally. 
Then,  
$\varprojlim_{N} H_d(U_{N,\l}^{(d)},\wh\Z)$ is a free module of rank $d+1$ over $\L:=\wh\Z[[\wh\Z(1)^{d+1}]]$. 
\end{conj}

Note that when $d=1$, $U_{N,\l}^{(1)}$ is isomorphic to the open subvariety of $X_{N,\l}$ defined by $x_1y_1\ne 0$ (see  \eqref{e1.1} for the notation). 
If the conjecture is true, one can define as before a cocycle $G_k \to \GL_{d+1}(\L)$ for each $\l \in L(k)$, and as its trace the adelic hypergeometric function with $d+1$ numerator parameters and argument $\l$. 

The Betti and de Rham cohomology of $U_{N,\l}^{(d)}$ is computed in \cite[Section 3]{asakura}.
The function $\FFF{d+1}{d}{a_0,\dots, a_d}{1,\dots, 1}{\l}$ with $a_i \in N^{-1}\Z$ appears as a period of $U_{N,\l}^{(d)}$, and the differential equation satisfied by this function is an eigencomponent with respect to the $\m_N^{d+1}$-action of the Gauss-Manin connection of the family $U_{N,\l}^{(d)}$ over $L$. 
On the other hand, one proves as in Section \ref{s4.2} that the number of rational points on $U_{N,\l}^{(d)}$ over a finite field containing $\m_N$ is expressed in terms of the corresponding hypergeometric functions over the finite field.


\end{document}